\documentclass[11pt]{article} 

\usepackage{amsmath}
\usepackage{lmodern}
\usepackage[T1]{fontenc}
\usepackage[babel=true]{microtype}
\usepackage{amsxtra}%
\usepackage{amsfonts}%
\usepackage{amssymb}%
\usepackage{amsthm}
\usepackage{graphicx}
\usepackage{epstopdf}
\usepackage{color,hyperref}
\hypersetup{colorlinks,breaklinks,
	linkcolor=blue,urlcolor=blue,
	anchorcolor=blue,citecolor=blue}

\usepackage{subcaption}
\usepackage{appendix}
\usepackage[font=small]{caption}
\usepackage{calc}

\usepackage[margin=0.75in]{geometry}

\newcommand{\eps}{\varepsilon}

\newcommand{\R}{\mathbb{R}}

\newcommand{\C}{\mathbb{C}}

\newcommand{\F}{\mathbb{F}}

\renewcommand{\phi}{\varphi}

\usepackage{enumitem}
\setitemize{noitemsep,topsep=0pt,parsep=2pt,partopsep=0pt}
\renewcommand{\Re}{\mathrm{Re} \,}

\def\XXint#1#2#3{{\setbox0=\hbox{$#1{#2#3}{\int}$ }
		\vcenter{\hbox{$#2#3$ }}\kern-.6\wd0}}

\newtheorem{thm}{Theorem}

\newtheorem*{thm*}{Theorem}

\newtheorem{prop}{Proposition}
\newtheorem{lemma}[prop]{Lemma}
\newtheorem{corollary}[prop]{Corollary}

\newtheorem{defi}{Definition}

\newtheorem{remark}[prop]{Remark}

\numberwithin{equation}{section}
\numberwithin{prop}{section}

\renewcommand{\u}{\mathbf{u}}

\newcommand{\Arg}{\mathrm{Arg} \,}

\newcommand{\utc}{\mathbf{u}_\mathrm{tc}}

\newcommand{\dtc}{d_\mathrm{tc}}

\newcommand{\cpc}{c_\mathrm{pc}}

\newcommand{\upcf}{u_\mathrm{pc}}
\newcommand{\vpcf}{v_\mathrm{pc}}
\newcommand{\uscf}{u_\mathrm{sc}}
\newcommand{\vscf}{v_\mathrm{sc}}

\newcommand{\Uscf}{\mathbf{u}_\mathrm{sc}}
\newcommand{\Upcf}{\mathbf{u}_\mathrm{pc}}

\newcommand{\Ascf}{\mathcal{A}_\mathrm{sc}}
\newcommand{\Lscf}{\mathcal{L}_\mathrm{sc}}

\newcommand\extrafootertext[1]{%
	\bgroup
	\renewcommand\thefootnote{\fnsymbol{footnote}}%
	\renewcommand\thempfootnote{\fnsymbol{mpfootnote}}%
	\footnotetext[0]{#1}%
	\egroup
}

\newcommand{\Apcf}{\mathcal{A}_\mathrm{pc}}
\newcommand{\Lpcf}{\mathcal{A}_\mathrm{pc}}

\newcommand{\ukpp}{u_\mathrm{kpp}}
\newcommand{\akpp}{a_\mathrm{kpp}}
\newcommand{\uckpp}{u_c^\mathrm{kpp}}
\newcommand{\Akpp}{\mathcal{A}_\mathrm{kpp}}

\newcommand{\vpt}{v_\mathrm{pt}}
\newcommand{\sigmapt}{\sigma_\mathrm{pt}}
\newcommand{\sigmasc}{\sigma_\mathrm{sc}}
\newcommand{\etasc}{\eta_\mathrm{sc}}
\newcommand{\etapc}{\eta_\mathrm{pc}}
\newcommand{\sigmapc}{\sigma_\mathrm{pc}}
\newcommand{\nusc}{\nu_\mathrm{sc}}
\renewcommand{\csc}{c_\mathrm{sc}}

\newcommand{\omegasc}{\omega_\mathrm{sc}}
\newcommand{\omegapc}{\omega_\mathrm{pc}}
\newcommand{\cpt}{c_\mathrm{pt}}

\newcommand{\Bsc}{\mathcal{B}_\mathrm{sc}}

\begin{document}
	\begin{center}
		{\fontsize{15}{15}\fontseries{b}\selectfont{Growth of cancer stem cell driven tumors: staged invasion, linear determinacy, and the tumor invasion paradox}}\\[0.2in]
		Montie Avery \\[0.1in]
		
		\textit{\footnotesize 
			Department of Mathematics and Statistics, Boston University, 665 Commonwealth Ave, Boston, MA 02215, USA} \\
	\end{center}
	
	\begin{abstract}
		We study growth of solid tumors in a partial differential equation model introduced by Hillen et al \cite{Hillen1} for the interaction between tumor cells (TCs) and cancer stem cells (CSCs). We find that invasion into the cancer-free state may be separated into two regimes, depending on the death rate of tumor cells. In the first, \emph{staged invasion regime}, invasion into the cancer-free state is lead by tumor cells, which are then subsequently invaded at a slower speed by cancer stem cells. In the second, \emph{TC extinction regime}, cancer stem cells directly invade the cancer-free state. Relying on recent results establishing front selection propagation under marginal stability assumptions, we use geometric singular perturbation theory to establish existence and selection properties of front solutions which describe both the primary and secondary invasion processes. With rigorous predictions for the invasion speeds, we are then able to heuristically predict how the total cancer mass as a function of time depends on the TC death rate, finding in some situations a \emph{tumor invasion paradox}, in which increasing the TC death rate leads to an \emph{increase} in the total cancer mass. Our methods give a general approach for verifying linear determinacy of spreading speeds of invasion fronts in systems with fast-slow structure. 
	\end{abstract}
	
	\section{Introduction}
	
	In recent decades, evidence has accumulated suggesting that a distinguished class of cells referred to as \emph{cancer stem cells} (CSCs) play a key role in the persistence, growth, and re-emergence of tumors \cite{Dingli, Gupta, Hanahan}. Cancer stem cells are distinguished by their ability to reproduce indefinitely, and to reproduce into both stem cells and non-stem tumor cells (TCs) \cite{Gupta}. Evidence for the existence of cancer stem cells was found first in leukemias \cite{Lapidot, Bonnet} and has since been found in many solid tumors; see for instance \cite{Singh1, Singh2, AlHajj, Dick, Fioriti, Maitland, Todaro, Cammareri}. Understanding the role cancer stem cells play in tumor growth and response to treatment is increasingly being recognized as an important component of developing cancer treatments \cite{Dingli}. 
	
	The following model for dynamics of solid tumors driven by cancer stem cells, including spatial diffusion, was introduced in \cite{Hillen1}:
	\begin{align}
		u_t &= D u_{yy} + p_s \gamma_u (1-u-v) u \nonumber \\
		v_t &= D v_{yy} + (1-p_s) \gamma_u  (1-u-v) u + \gamma_v (1-u-v) v - \alpha v, \quad y \in \R, \quad t > 0 \label{e: model}
	\end{align}
	The variable $u(y,t)$ denotes the concentration of cancer stem cells (CSCs), while $v(y,t)$ denotes the concentration of tumor cells (TCs). In the microscopic dynamics associated to this model model, CSCs reproduce at rate $\gamma_u$ by dividing either into two CSCs with probability $p_s$, or one CSC and one tumor stem cell, with probability $1-p_s$. Tumor cells divide only into other tumor cells, with rate $\gamma_v$. Both species of cells diffuse with diffusion coefficient $D$. Tumor cells die  with rate $\alpha$ due to a combination of their inherent limited life span and possibly an external treatment regimen. The factor of $(1-u-v)$ multiplying the growth terms reflects the fact that TCs and CSCs compete for resources. In particular, the carrying capacities of TCs and CSCs are assumed to be the same.  For more background on this model, see \cite{Hillen1, Hillen2}, the latter of which also compares the dynamics of this model to an agent-based model for the cell dynamics. 
	
	As in \cite{Hillen2} we consider \eqref{e: model} with $p_s = \eps > 0$ small, to reflect the observation that CSCs typically produce non-stem tumor cells at a much higher rate than they produce new CSCs. We also assume, as in \cite{Hillen2}, that cell diffusion is slow and on the same time scale, so that we set $D = \eps d$ with $d > 0$ fixed. For simplicity, we assume that the overall reproduction rates of CSCs and TCs are the same, setting $\gamma_u = \gamma_v = 1$. Introducing the rescaled time and space variables $\tau = \eps t, x = \frac{y}{\sqrt{d}}$, we find the rescaled model
	\begin{align}
		u_\tau &=  u_{xx} + (1-u-v)u \nonumber \\
		\eps v_\tau &= \eps  v_{xx} + (1-\eps) (1-u-v) u + (1-u-v) v - \alpha v. \label{e: eqn}
	\end{align}
	We study tumor growth in this model through the mathematical framework of front propagation into unstable states. The system \eqref{e: eqn} admits three spatially constant equilibria: the \emph{cancer-free state} $(u,v) = (0,0)$, the \emph{pure TC state} $(u, v) = (0, 1-\alpha)$, and the \emph{pure CSC state} $(u,v) = (1,0)$. Since the cell population density should be positive, the pure TC state is only physically relevant in the regime $0 < \alpha < 1$. The cancer-free and pure TC states are both unstable (in the regime $0 < \alpha < 1$ in which the pure TC state is relevant), while the pure CSC state is stable. One then expects that when perturbed, the cancer-free state is invaded by some combination of TC cells and CSC cells. We are interested in understanding the nature of this invasion process and predicting the associated invasion speed. 
	
	\begin{figure}
		\centering
		\begin{subfigure}{0.495\textwidth}
			\includegraphics[width=1\textwidth]{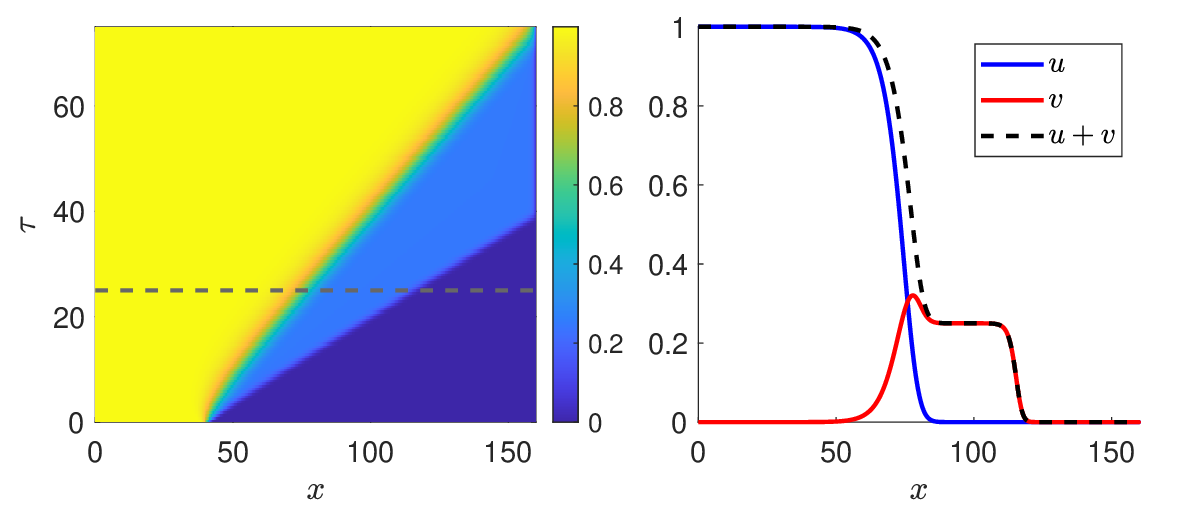}
		\end{subfigure}
		\hfill
		\begin{subfigure}{0.495\textwidth}
			\includegraphics[width=1\textwidth]{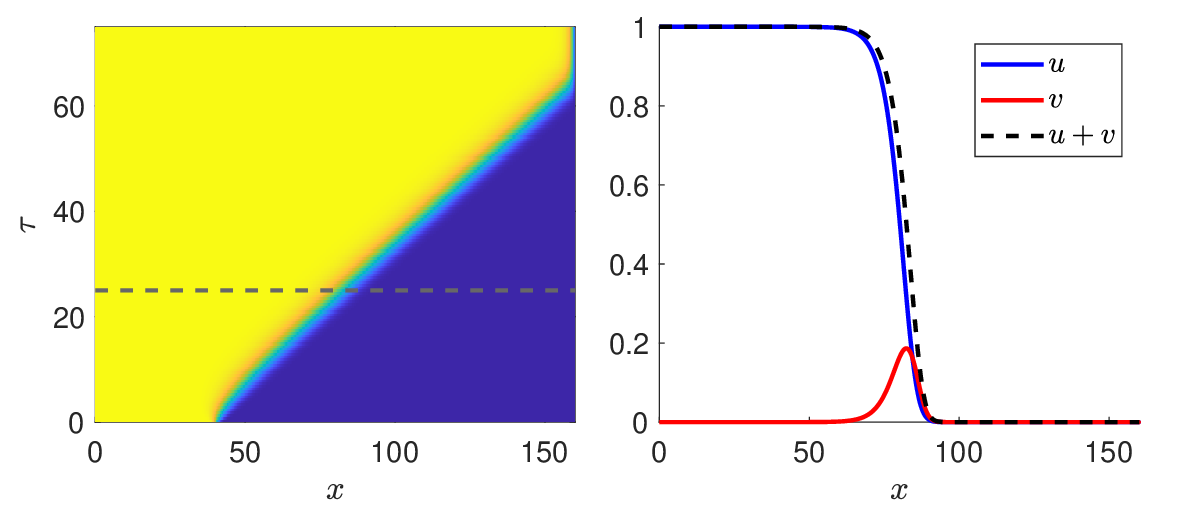}
		\end{subfigure}
		\caption{Far left: space time plot of $u + v$ for $\alpha = 0.75, \eps = 0.1$. Center left: $u, v,$ and $u$ against $x$ for the slice $\tau = 25$ indicated in the space time plot with the dashed line, for $\alpha = 0.75, \eps = 0.1$. Middle right: space time plot of $u + v$ for $\alpha = 1.25, \eps = 0.1$. Far right: $u, v,$ and $u+v$ against $x$ for the slice $\tau = 25$ indicated in the space time plot with the dashed line, for $\alpha = 1.25, \eps = 0.1$.}
		\label{f: spacetime}
	\end{figure}
	
	\paragraph{Summary of main results.} We find that the dynamics may be separated into two distinct cases depending on the values of the parameters $\alpha$ and $\eps$. 
	
	In the \emph{staged invasion regime}, $0 < \alpha < \frac{1}{1+\eps}$, the cancer-free state is initially invaded by tumor cells, with primary invasion speed
	\begin{align}
		c_\mathrm{pt}(\alpha, \eps) = 2 \sqrt{\frac{1-\alpha}{\eps}}.
	\end{align}
	Since the pure TC state is unstable against CSCs, the tumor cells in the wake of the primary invasion front are then themselves invaded by cancer stem cells, with secondary invasion speed $\csc(\alpha) = 2 \sqrt{\alpha} < c_\mathrm{pt}(\alpha; \eps)$. The primary invasion of the cancer-free state by tumor cells is described rigorously in Theorem \ref{t: primary invasion}, while the subsequent invasion of the tumor cells by cancer stem cells is described in Theorem \ref{t: secondary invasion}. See the left two panels of Figure \ref{f: spacetime} for plots of solutions of \eqref{e: eqn} in this regime. 
	
	In the \emph{TC extinction regime}, $\alpha > \frac{1}{1+\eps}$, the death rate of tumor cells is too high to support primary invasion by TCs, and so the cancer-free state is instead invaded directly by cancer stem cells, with speed $\cpc = 2$. This invasion process is described rigorously in Theorem \ref{t: primary CSC invasion}. See the right two panels of Figure \ref{f: spacetime} for plots of solutions in this regime. 

	The challenge with rigorously describing these invasion processes is that \eqref{e: eqn} does not have a comparison principle, the key tool typically used to prove results on front propagation. Indeed, compared for instance to competitive Lotka-Volterra systems, the term $(1-\eps) (1-u-v)u$ in \eqref{e: eqn}, representing the assumption that CSCs may reproduce into TCs, breaks the competitive structure and associated comparison principles. We therefore rely instead on the recent conceptual approach to front propagation developed in \cite{AveryScheelSelection, AverySelectionRD}, which abandons comparison principles and instead establishes invasion results relying only on existence and marginal spectral stability of invasion fronts. These existence and stability conditions in \eqref{e: eqn} may be formulated as ODE problems, and so we can use Fenichel's geometric singular perturbation theory \cite{Fenichel} to study the singular $\eps = 0$ limit. 
	
	The \emph{tumor invasion paradox} refers to a phenomenon in which increasing the death rate of TCs --- for instance, by increasing the intensity of an applied treatment --- may lead to faster spatial spreading of the tumor. This phenomenon has been observed, for instance, in response to radiation therapies, where it is referred to as radiation-induced invasion \cite{Maggiorella, Kargiotis, Lin}.
	
	Within the model \eqref{e: eqn}, the authors of \cite{Hillen2} establish spreading of CSCs into the pure TC state rigorously for $\eps = 0, 0 < \alpha < 1$, and associate the tumor invasion paradox with the fact that this spreading speed is increasing in the tumor death rate $\alpha$. Here, since we extend the analysis to $\eps > 0$, we obtain a more refined picture of the dynamics, in particular observing the staged invasion for $0 < \alpha < 1$. Whether there is a ``paradox'' or not then depends on what quantities characterizing the tumor spread one is most interested in. We find, for $0 < \alpha < 1$: 
	\begin{itemize}
		\item The speed of the primary invasion front, describing the invasion of TCs into the cancer-free state, is always decreasing in $\alpha$. So, increasing the tumor death rate always slows down the leading edge of cancer cells. 
		\item The speed of the secondary invasion front, describing the invasion of TCs by CSCs, is increasing in $\alpha$. That is, increasing the tumor cell death rate enhances the spread of CSCs. 
		\item In a large bounded domain $x \in (0,L)$, fixing the initial condition and the time $\tau> 0$, the total cancer mass at time $\tau$ is decreasing in $\alpha$ unless the tumor cells have already spread to the entire domain, after which the mass starts to increase in $\alpha$.
	\end{itemize}
	The first two points are established rigorously by Theorems \ref{t: primary invasion} and \ref{t: secondary invasion}. To establish the third point, in Section \ref{s: mass} we argue heuristically using the rigorously established front speeds, and corroborate with numerical simulations. So, if one is most interested in tracking the leading edge of cancer cells, then there is no paradox. However, CSCs are especially important for the resurgence of tumors, so one may consider the spreading speed of CSCs to be the most important quantity, in which case the second point becomes the most relevant. Tracking the total tumor mass highlights the relevance of transient dynamics for this staged invasion process.  
	
	\subsection{Rigorous results}
	
	 We now introduce some definitions needed to precisely state our main results.
	\begin{defi}[Primary TC front]
		A \emph{primary tumor cell producing front} (primary TC front) with speed $c$ is a traveling wave solution $(u(x,t), v(x,t)) = (0, \vpt(x-ct))$ to \eqref{e: eqn} satisfying
		\begin{align}
			\lim_{\xi \to -\infty}  \vpt(\xi) = 1-\alpha, \quad \lim_{\xi \to \infty} \vpt(\xi) = 0. 
		\end{align}
	\end{defi}
		\begin{defi}[Primary CSC front]
		A \emph{primary cancer stem cell producing front} (primary CSC front) with speed $c$ is a traveling wave solution $(u(x,t), v(x,t)) = (\upcf(x-ct), \vpcf(x-ct))$ to \eqref{e: eqn} satisfying 
		\begin{align}
			\lim_{\xi \to -\infty} (\upcf(\xi), \vpcf(\xi)) = (1, 0), \quad \lim_{\xi \to \infty} (\upcf(\xi), \vpcf(\xi)) = (0, 0). 
		\end{align}
	\end{defi}
	\begin{defi}[Secondary CSC front]
		A \emph{secondary cancer stem cell producing front} (secondary CSC front) with speed $c$ is a traveling wave solution $(u(x,t), v(x,t)) = (\uscf(x-ct), \vscf(x-ct))$ to \eqref{e: eqn} satisfying
		\begin{align}
			\lim_{\xi \to -\infty} (\uscf(\xi), \vscf(\xi)) = (1, 0), \quad \lim_{\xi \to \infty} (\uscf(\xi), \vscf(\xi)) = (0, 1- \alpha). 
		\end{align}
	\end{defi}

	Primary TC fronts describe the invasion of the cancer-free state by the pure TC state, which occurs for $0 < \alpha < 1$. Since the pure TC state is unstable for $0 < \alpha < 1$, the pure CSC state then begins to invade the pure TC state, developing a secondary CSC front. Primary CSC fronts are relevant in the TC extinction regime $\alpha > 1$. 
	
	Notice that $u \equiv 0$ is an invariant subspace in \eqref{e: eqn} (reflecting the fact that TCs do not produce CSCs), and the pure TC state $v = 1- \alpha$ is stable within this invariant subspace.  We therefore expect that, within this invariant subspace, localized perturbations to the cancer-free state $v \equiv 0$ will spread outward and grow into stable invasion fronts propagating with a selected speed. As is often the case in the mathematical study of front invasion, we focus on one front interface only, thereby replacing localized perturbations with perturbations supported on a half-line, reflected in the following definition. 
	
	\begin{defi}[Initial data for primary TC invasion]
		We say that the pair $(u_0, v_0) \in L^\infty(\R, \R^2)$ is \emph{primary TC invasion initial data} if the following hold. 
		\begin{enumerate}
			\item We have 
			\begin{align}
				\lim_{\xi \to -\infty} u_0(\xi) = 0, \quad \lim_{\xi \to -\infty} v_0(\xi) = 1-\alpha.
			\end{align}
			\item There exists $R > 0$ such that $u_0 (\xi) = v_0(\xi) = 0$ for all $\xi \geq R$
		\end{enumerate} 
	\end{defi}
	
	\begin{defi}[Initial data for primary CSC invasion]
		We say that the pair $(u_0, v_0) \in L^\infty (\R, \R^2)$ is \emph{primary CSC invasion initial data} if the following hold.
		\begin{itemize}
			\item We have
			\begin{align}
				\lim_{\xi \to \infty} u_0(\xi) = 1, \quad \lim_{\xi \to -\infty} v_0(\xi) = 0.
			\end{align}
			\item There exists $R > 0$ such that $u_0 (\xi) = v_0(\xi) = 0$ for all $\xi > R$. 
		\end{itemize}
	\end{defi}
	Primary invasion initial data therefore describe the dynamics of propagation into the cancer-free state, which may be lead by either TCs or CSCs. Since the pure TC state is unstable for $0 < \alpha < 1$, we also consider invasion into the pure TC state. 
	\begin{defi}[Initial data for secondary CSC invasion]
		We say that the pair $(u_0, v_0) \in L^\infty(\R, \R^2)$ is \emph{secondary CSC invasion initial data} if the following hold.
		\begin{itemize}
			\item We have
			\begin{align}
				\lim_{\xi \to -\infty} u_0 (\xi) = 1, \quad \lim_{\xi \to -\infty} v_0(\xi) = 0. 
			\end{align}
			\item There exists $R > 0$ such that $u_0 (\xi) = 0$ and $v_0(\xi) = 1- \alpha$ for all $\xi > R$. 
		\end{itemize}
	\end{defi}

	With these definitions in hand, we are now ready to rigorously formulate our results on existence and selection of invasion fronts.
	
	\begin{thm}[Primary TC invasion]\label{t: primary invasion}
		Assume $0 < \alpha < 1$ and $\eps > 0$. Let $(u(x,t), v(x,t))$ be a solution to \eqref{e: eqn} with any primary TC invasion initial data additionally satisfying $u_0 (x) \equiv 0$ and $v_0 \geq 0$. Then $u(x,t) \equiv 0$ for all time, and there exists a shift $\sigma_\mathrm{pt}(t) \in \R$ such that
		\begin{align}
			\lim_{t \to \infty} v(x + \sigma_\mathrm{pt} (t), t) = \vpt^*(x; \alpha, \eps), \label{e: primary TC convergence}
		\end{align}
		for all $x \in \R$, where $\vpt^* (x; \alpha, \eps)$ is a primary TC front with speed 
		\begin{align}
			c_\mathrm{pt} (\alpha, \eps) = 2 \sqrt{ \frac{1-\alpha}{\eps}}. 
		\end{align}
		As $t \to \infty$, the shift function $\sigmapt(t)$ has asymptotics
		\begin{align}
			\sigma_\mathrm{pt}(t) = c_\mathrm{pt} (\alpha, \eps) t - \frac{3}{2 \eta_\mathrm{pt} (\alpha, \eps)} \log t + x_\infty + \mathrm{o}(1),
		\end{align}
		where $x_\infty \in \R$ depends on $v_0$, and $\eta_\mathrm{pt} (\alpha, \eps) = \frac{1}{2} c_\mathrm{pt} (\alpha, \eps)$. 
		
		In particular, the solution propagates with asymptotic speed $c_\mathrm{pt}(\alpha, \eps)$. The front $\vpt^*$ is unstable against perturbations with $u_0 (x) > 0$ on some set of positive measure. 
	\end{thm}
	When $u_0 \equiv 0$, \eqref{e: eqn} reduces to the scalar reaction-diffusion equation
	\begin{align}
		\eps v_\tau = \eps v_{xx} + (1- \alpha) v - v^2,
	\end{align}
	which is a Fisher-KPP equation. Theorem \ref{t: primary invasion} therefore follows readily from classical results of Bramson \cite{Bramson1, Bramson2}, with recent refinements via PDE techniques \cite{Comparison1, Comparison2, Comparison3}. In particular, recent results also establish higher asymptotics for $\sigma_\mathrm{pt}(t)$ as well as precise convergence rates in \eqref{e: primary TC convergence} \cite{Comparison3, Graham, AnHendersonRyzhik}. 
	
	The next two theorems describe propagation of CSCs. In the regime $0 < \alpha < 1$, CSCs invade the TC state with speed $\csc (\alpha) = 2 \sqrt{\alpha}$.
	\begin{thm}[Secondary CSC invasion]\label{t: secondary invasion}
		Fix $0 < \alpha < 1$, and assume $\eps > 0$ is sufficiently small. Given any $\delta > 0$, there exist secondary CSC invasion initial data $(u_0 (x), v_0(x))$ and an associated shift $\sigmasc(t) \in \R$ such that the solution $(u(x,t), v(x,t))$ to \eqref{e: eqn} with initial data $(u_0, v_0)$ satisfies
		\begin{align}
			\sup_{x \in \R} \left| \begin{pmatrix} u(x + \sigmasc(t),t) \\ v(x+\sigmasc(t),t) \end{pmatrix} - \begin{pmatrix} \uscf^*(x; \alpha, \eps) \\ \vscf^*(x; \alpha, \eps) \end{pmatrix} \right| < \delta \label{e: theorem 2 closeness}
		\end{align}
		for all $t$ sufficiently large, where $(\uscf^* (x; \alpha, \eps), \vscf^*(x; \alpha, \eps))$ is a secondary CSC front with speed $\csc(\alpha) = 2 \sqrt{\alpha}$. 
		As $t \to \infty$, the shift function $\sigmasc(t)$ has the asymptotics
		\begin{align}
			\sigmasc(t) = \csc (\alpha) t - \frac{3}{2 \etasc (\alpha)} \log t + x_\infty + \mathrm{o}(1),
		\end{align}
		where $x_\infty$ depends on the initial data, and $\etasc(\alpha) = \frac{1}{2} \csc (\alpha) = \sqrt{\alpha}$. In particular, the pure CSC state invades the pure TC state with asymptotic speed $\csc(\alpha) = 2 \sqrt{\alpha}$. 
	\end{thm}
	\begin{remark}
		We actually obtain a stronger result than that stated in Theorem \ref{t: secondary invasion}, which we have simplified here for clarify of presentation. Namely, the closeness to the secondary CSC front \eqref{e: theorem 2 closeness} holds in a much stronger norm, and the secondary invasion initial data used in the statement of Theorem \ref{t: secondary invasion} belongs to a class which is open in a natural topology and which all leads to propagation with speed $\csc(\alpha)$; see \cite[Theorem 1]{AveryScheelSelection} for further details. 
	\end{remark}
	
	In the TC extinction regime $\alpha > 1$, CSCs directly invade the cancer free state with speed $\cpc = 2$. 
	\begin{thm}[Primary CSC invasion in TC extinction regime]\label{t: primary CSC invasion}
		Fix $\alpha >1$ and assume $\eps > 0$ is sufficiently small. Given any $\delta > 0$, there exists primary CSC invasion initial data such that the results of Theorem \ref{t: secondary invasion} hold with $(\uscf^*, \vscf^*)$ replaced by a primary CSC front $(\upcf^*, \vpcf^*)$, with the spreading speed now given by $\cpc = 2$, and with the shift function now given by
		\begin{align}
			\sigmapc(t) = \cpc  t - \frac{3}{2 \etapc } \log t + x_\infty + \mathrm{o}(1),
		\end{align}
		where $\etapc = 1$. In particular, in this regime the pure CSC state invades the cancer-free state with asymptotic speed $\cpc = 2$. 
	\end{thm}
	
	The nonlinear invasion speeds identified in Theorems \ref{t: primary invasion} through \ref{t: primary CSC invasion} are each the linear spreading speeds associated to the linearization about the respective unstable state. Hence, Theorems \ref{t: primary invasion} through \ref{t: primary CSC invasion} establish that in the parameter regimes considered here, the invasion processes are linearly determined, or \emph{pulled} \cite{vanSaarloos, AveryScheelSelection}. 
	
	Comparing Theorems \ref{t: primary invasion} and \ref{t: secondary invasion}, we see that Theorem \ref{t: primary invasion} gives a stronger description of the resulting invasion process, applying to a larger class of initial data and guaranteeing convergence to a TC front, in an appropriate frame. This global convergence is possible due to the presence of a comparison principle in the $u \equiv 0$ subspace. By contrast, Theorem \ref{t: secondary invasion} applies to smaller sets of initial data, and only establishes that the solution remains close to the CSC front for all time, rather than showing full convergence. We emphasize, however, that the initial data used in Theorem \ref{t: secondary invasion} is not unnaturally constructed, and any sufficiently small perturbations of the same initial data lead to the same result. Furthermore, we expect that the methods of \cite{AveryScheelSelection, AverySelectionRD} could be extended to improve \eqref{e: theorem 2 closeness} to full convergence to the front, although we do not pursue it here. Hence, Theorem \ref{t: secondary invasion} still gives a strong rigorous foundation for predicting the invasion speed into the pure TC state in \eqref{e: eqn}. 
	
	\begin{remark}
		In the introduction, we identified $0 < \alpha < \frac{1}{1+\eps}$ as the staged invasion regime, and $\alpha > \frac{1}{1+\eps}$ as the TC extinction regime. We phrase the rigorous results of Theorems \ref{t: secondary invasion} and \ref{t: primary CSC invasion} slightly differently, for instance first fixing $0 < \alpha < 1$ and then choosing $0 < \eps < \eps_*(\alpha)$ sufficiently small. The threshold $\alpha = \frac{1}{1+\eps}$ is obtained by setting $\csc(\alpha) = \cpt(\alpha, \eps)$, and numerical simulations corroborate that we observe staged invasion for $0 < \alpha < \frac{1}{1+\eps}$ and primary CSC invasion for $\alpha > \frac{1}{1+\eps}$. 
	\end{remark}
	
	\paragraph{Outline.} The remainder of the paper is organized as follows. In Section \ref{s: framework}, we review the framework for front selection developed in \cite{AveryScheelSelection, AverySelectionRD} and translate the assumptions of \cite{AverySelectionRD} into results to be proven here, reducing the proofs of Theorems \ref{t: secondary invasion} and \ref{t: primary CSC invasion} to ODE problems for existence and spectral stability of fronts. In Section \ref{s: existence}, we use geometric singular perturbation theory to reduce the traveling wave equation to a regularized problem on a slow manifold, and use far-field/core decompositions which explicitly capture asymptotics in the leading edge to continue front solutions from the $\eps = 0$ limit. In Section \ref{s: spectral} we again use geometric singular perturbation theory to regularize the associated eigenvalue problem, and use a far-field/core decomposition to exclude eigenvalues bifurcating from the essential spectrum. In Section \ref{s: mass}, we discuss heuristics for computing the total cancer mass as a function of time, compare to numerics, and discuss implications for the tumor invasion paradox. We conclude in Section \ref{s: discussion} with a summary and discussion of extensions of the model. 
	\section{Front selection via marginal stability}\label{s: framework}
	The program for establishing pulled front propagation in \cite{AveryScheelSelection, AverySelectionRD} may be summarized in the following steps.
	\begin{enumerate}
		\item Compute the linear spreading speed associated to the unstable state being invaded.
		\item Verify that there exists a traveling front solution propagating at this linear spreading speed, with generic asymptotics. 
		\item Verify that this traveling front is marginally spectrally stable in an appropriate exponentially weighted space. 
	\end{enumerate}
	These steps are encoded as Hypotheses 1-4 in \cite{AveryScheelSelection} for higher order scalar parabolic equations, and as Hypotheses 1-4 in \cite{AverySelectionRD} for multi-component reaction diffusion systems as considered here. We now formulate these hypotheses in the present context as results to be proven. 
	
	Dividing the second equation by $\eps$, we may rewrite the system \eqref{e: eqn} as 
	\begin{align}
		\u_t = \u_{xx} + F(\u; \alpha, \eps), \quad x \in \R, \quad  t > 0, \label{e: sys divided by eps}
	\end{align}
	where 
	\begin{align*}
		\u = \begin{pmatrix} u \\ v \end{pmatrix}, \quad F(\u; \alpha, \eps) = \begin{pmatrix} (1-u-v)u \\ \frac{1}{\eps} \left[ (1-\eps) (1-u-v)u + (1-u-v)v - \alpha v \right] \end{pmatrix}.
	\end{align*}
	Using this formulation, we now explain how to prove Theorems \ref{t: secondary invasion} and \ref{t: primary CSC invasion} through the framework of \cite{AveryScheelSelection, AverySelectionRD}. 
	
	\subsection{Invasion of TCs by CSCs --- formulation via marginal stability}\label{s: secondary csc formulation}
	We first formulate results establishing the linear spreading speed as well as existence and spectral stability of secondary CSC fronts. The results are analogous for primary CSC invasion, which occurs for $\alpha > 1$, and we summarize them in Section \ref{s: primary CSC formulation}. 
	
	\paragraph{Linear spreading speed.} The speed $\csc(\alpha) = 2 \sqrt{\alpha}$ identified in Theorem \ref{t: secondary invasion} is the \emph{linear spreading speed} associated to the rest state $\u_\mathrm{tc} = (0, 1-\alpha)$, which is the speed at which the level set $\{ u = 1 \}$ propagates in the linearized equation,
	\begin{align}
		\u_t = \u_{xx}+ F'(\u_\mathrm{tc}; \alpha, \eps) \u. \label{e: tc linearized}
	\end{align}
	It has long been known that linear spreading speeds may be calculated via a so-called \emph{pinch point analysis} \cite{vanSaarloos}, with a recent reformulation in \cite{HolzerScheelPointwiseGrowth} on which we rely here. 
	
	To compute the linear spreading speed, we therefore introduce the \emph{dispersion relation}
	\begin{align}
		d_\mathrm{tc}(\lambda, \nu; c, \alpha, \eps) &= \det (\nu^2 + c \nu I + F'(\u_\mathrm{tc}; \alpha, \eps) - \lambda I) \nonumber \\
		&= \det \begin{pmatrix}
			\nu^2 + c \nu + \alpha - \lambda & 0 \\
			\frac{1}{\eps} [(2-\eps) \alpha - 1] & \nu^2 + c \nu + \frac{1}{\eps} (\alpha - 1) - \lambda
		\end{pmatrix}. \label{e: tc dispersion} 
	\end{align}
	obtained by considering \eqref{e: tc linearized} in the moving frame with arbitrary speed $c$, and substituting the Fourier-Laplace ansatz $\u (x,t) \sim e^{\lambda t} e^{\nu x}$. The linear spreading speed is associated to double (in $\nu$) roots of the dispersion relation with $\lambda = 0$ which satisfy a so-called pinching condition; see \cite{HolzerScheelPointwiseGrowth} for background. 
	
	\begin{lemma}[Linear spreading speed into pure TC state.]\label{l: tc linear spreading speed}
		Fix $0 < \alpha < 1$. There exists $\eps_*(\alpha) > 0$ such that if $0 < \eps < \eps_*(\alpha)$, then with $c = \csc (\alpha) = 2 \sqrt{\alpha}$ and $\etasc = \sqrt{\alpha}$, the dispersion relation \eqref{e: tc dispersion} satisfies the following.
		\begin{enumerate}
			\item (Pinched double root) For $\tilde{\nu}, \lambda$ small, we have the expansion
			\begin{align}
				\dtc (\lambda, -\etasc(\alpha) + \tilde{\nu}; \csc(\alpha), \alpha, \eps) = \dtc^{10} \lambda + \dtc^{02} \tilde{\nu}^2 + \mathrm{O} (\lambda^2, \lambda \tilde{\nu}, \tilde{\nu}^3) \label{e: tc double root}
			\end{align}
			where $\dtc^{10}(\alpha, \eps) > 0$ and $\dtc^{02} (\alpha, \eps) < 0$. 
			\item (Minimal critical spectrum) If $\dtc(i \omega, -\etasc(\alpha) + ik; \csc(\alpha), \alpha, \eps) = 0$ for some $\omega, k \in \R$, then $\omega = k = 0$. 
			\item (No unstable essential spectrum) There are no solutions to $\dtc(\lambda, -\etasc(\alpha) + ik; \csc(\alpha), \alpha, \eps) = 0$ with $\Re \lambda > 0$ and $k \in \R$. 
		\end{enumerate}
		Together with the results of \cite{HolzerScheelPointwiseGrowth}, this implies that the linear spreading speed of \eqref{e: tc linearized} is $c = \csc(\alpha)$. 
	\end{lemma}
	\begin{proof}
		This follows from a direct calculation, noticing that the dispersion relation factors as 
		\begin{align}
			\dtc(\lambda, \nu; c, \alpha, \eps) = (d \nu^2 + c \nu + \alpha - \lambda)(d \nu^2 + c \nu + \frac{1}{\eps} (\alpha-1) - \lambda). 
		\end{align}
		The speed $\csc(\alpha) = 2 \sqrt{\alpha}$ is then the linear spreading speed into the $u \equiv 0$ state in the $u$ equation, in absence of coupling the the $v$ equation. Associated to this linear spreading speed is the exponential rate $\nusc (\alpha) = - \sqrt{\alpha} = -\etasc(\alpha)$. Restricting to perturbations which are exponentially localized with rate $e^{-\etasc (\alpha) \xi}$ stabilizes the essential spectrum in the $u$ component, leaving it touching the imaginary axis only at the origin. The second two conditions of Lemma \ref{l: tc linear spreading speed}, again verifiable by direct calculation, guarantee that no other instabilities arise in this weighted space under coupling to the $v$ equation, so that this speed persists as the linear spreading speed. 
	\end{proof}
	
	\paragraph{Front existence with generic asymptotics.} 
	Front solutions satisfy the traveling wave equation
	\begin{align}
		0 = \u_{\xi \xi}+ c \u_\xi + F(\u; \alpha, \eps). \label{e: tw equation}
	\end{align}
	Roots $\nu$ of the dispersion relation with $\lambda = 0$ correspond to eigenvalues of the linearization of the corresponding first-order system about the equilibrium corresponding to $\mathbf{u}_\mathrm{tc}$. The double root \eqref{e: tc double root} of the dispersion relation is then associated to a length 2 Jordan block in this linearization, which generically corresponds to weak exponential decay as captured in the following result. 
	\begin{prop}[Existence of critical secondary CSC front]\label{p: secondary CSC existence}
		Fix $0 < \alpha < 1$. There exists $\eps_*(\alpha) > 0$ such that if $0 < \eps < \eps_*(\alpha)$, then \eqref{e: tw equation} with speed $c = \csc(\alpha)$ admits a secondary CSC front solution $\Uscf (\xi; \alpha, \eps) = (\uscf(\xi; \alpha, \eps), \vscf(\xi; \alpha, \eps))^T$, with Jordan block-type asymptotics as $\xi \to \infty$, 
		\begin{align}
			\begin{pmatrix}
				\uscf(\xi; \alpha, \eps) \\ \vscf(\xi; \alpha, \eps)
			\end{pmatrix} = \begin{pmatrix} 0 \\ 1 - \alpha \end{pmatrix} + [\u_0 \xi + \u_1 + a \u_0] e^{\nusc(\alpha) \xi} + \mathrm{O}(e^{(\nusc(\alpha) - \eta)\xi}), \label{e: secondary csc front asymptotics}
		\end{align}
		for some $\eta > 0$, $\u_0, \u_1 \in \R^2$ and $a \in \R$. In the wake, we have exponential convergence,
		\begin{align}
			\begin{pmatrix}
				\uscf(\xi; \alpha, \eps) \\ \vscf(\xi; \alpha, \eps)
			\end{pmatrix} = \begin{pmatrix} 1 \\ 0 \end{pmatrix}
			+ \mathrm{O} (e^{\eta \xi}), 
		\end{align}
		as $\xi \to -\infty$, for some $\eta > 0$. Moreover, $\Uscf$ depends continuously on $\alpha$ and $\eps$, locally uniformly on $\R$, both $\uscf$ and $\vscf$ are positive, and we have the estimate $\uscf + \vscf > \frac{3 (1-\alpha)}{4}$. 
	\end{prop}
	We prove Proposition \ref{p: secondary CSC existence} in Section \ref{s: existence} by perturbing from the singular $\eps = 0$ limit, first using Fenichel's geometric singular perturbation theory to reduce to a regularized problem on an associated slow manifold in Section \ref{s: gspt}. Heteroclinic front solutions in the $\eps = 0$ limit may be found via classical phase space arguments. Showing that these heteroclinics persist for $\eps \neq 0$ is not difficult once the singular perturbation is regularized, since they lie in a transverse intersection of stable and unstable manifolds. The difficult part is to verify that the asymptotics \eqref{e: secondary csc front asymptotics}, which are essential to the spreading dynamics \cite{AveryScheelSelection, vanSaarloos}, persist. We do this by using a far-field/core decomposition which captures leading edge asymptotics via an explicit ansatz, leaving us to solve for an exponentially localized core correction via Fredholm theory and the implicit function theorem, a strategy used to prove persistence of pulled fronts in \cite{AveryGarenaux, avery22, AverySelectionRD}. 
	
	\paragraph{Spectral stability.} Let $A : D(A) \subseteq X \to X$ be a closed, densely defined operator on a Banach space $X$, and let $A^*$ denote its adjoint. The operator $A$ is said to be \emph{Fredholm} if the following conditions hold.
	\begin{itemize}
		\item $\dim \ker A$ and $\dim \ker A^*$ are finite.
		\item The range of $A$ is closed in $X$. 
	\end{itemize}
	The Fredholm index of $A$ is defined to be
	\begin{align}
		\text{ind} (A) = \dim \ker A - \dim \ker A^*. 
	\end{align}
	We say $\lambda \in \C$ is in the \emph{essential spectrum} of $A$ if the operator $A - \lambda$ is either not Fredholm, or is Fredholm with a non-zero index. We say $\lambda$ is in the \emph{point spectrum} of $A$ is $A - \lambda$ is Fredholm with index 0 but not invertible. 
	
	Let 
	\begin{align}
		\Ascf = \partial_\xi^2 + \csc (\alpha) \partial_\xi + F'(\Uscf(\xi; \alpha, \eps); \alpha, \eps) 
	\end{align}
	denote the linearization about a secondary CSC front from Proposition \ref{p: secondary CSC existence}. The coefficients of $\Ascf$ converge exponentially to constant limits as $x \to \pm \infty$. Using the definition above, the essential spectrum of an operator is readily seen to be invariant under compact perturbations and so can be computed from the limiting operators as $\xi = \pm \infty$. In particular, the essential spectrum of $\Ascf$ is unstable due to the instability of the background state $\utc$. 
	
	Spectral stability may be recovered by restricting to perturbations which decay exponentially as $\xi \to \infty$, with optimal rate $e^{-\etasc (\alpha) \xi}$. Hence, we introduce a smooth positive weight function $\omega_*(\xi; \alpha)$ satisfying
		\begin{align}
				\omegasc(\xi; \alpha) = \begin{cases}
						1, &\xi \leq -1, \\ 
						e^{\etasc (\alpha) \xi}, & \xi \geq 1,
					\end{cases} \label{e: omega def}
			\end{align}
		where $\etasc(\alpha) = \sqrt{\alpha}$. Restricting to exponentially localized perturbations is equivalent to considering the conjugated linearization
		\begin{align}
			\Lscf \u := \omegasc \Ascf \left(\frac{1}{\omegasc} \u \right). 
		\end{align}
	We recover the following marginal spectral stability of $\Lscf$. 
	\begin{prop}[Marginal stability of secondary CSC front]\label{p: secondary csc spectral stability}
		Fix $0 < \alpha < 1$. There exists $\eps_*(\alpha) > 0$ such that if $0 < \eps < \eps_*(\alpha)$, then the essential spectrum of $\Lscf : H^2 (\R, \C^2) \subset L^2(\R, \C^2) \to L^2(\R, \C^2)$ is marginally stable, touching the imaginary axis only at the origin via the curve $\{ -k^2 : k \in \R\}$, and otherwise contained in the left half plane. Moreover, $\Lscf$ has no eigenvalues $\lambda$ with $\Re \lambda \geq 0$, and there is no bounded solution to $\Lscf \u = 0$. 
	\end{prop}
	We prove Proposition \ref{p: secondary csc spectral stability} in Section \ref{p: secondary csc spectral stability} again by perturbing from the $\eps = 0$ limit. This singular limit can again be regularized by passing to a reduced problem on a slow manifold. The essential spectrum may be readily computed from the limiting operators via the Fourier transform. The most difficult part of the proof is to exclude eigenvalues bifurcating from the essential spectrum at the origin. We do this using a functional analytic approach to tracking eigenvalues near and through the essential spectrum developed in \cite{PoganScheel}, an alternative to earlier geometric arguments known collectively as the gap lemma \cite{GardnerZumbrun, KapitulaSandstede}. 
	
	\begin{proof}[Proof of Theorem \ref{t: secondary invasion}]
		Lemma \ref{l: tc linear spreading speed}, Proposition \ref{p: secondary CSC existence}, and Proposition \ref{p: secondary csc spectral stability} precisely show that for $0 < \alpha < 1$ and $0 < \eps < \eps_*(\alpha)$, the system \eqref{e: eqn} satisfies \cite[Hypotheses 1 through 4]{AverySelectionRD}, immediately implying Theorem \ref{t: secondary invasion} by \cite[Theorem 1]{AverySelectionRD}. 
	\end{proof}
	Hence, establishing invasion at the linear spreading speed as stated in Theorem \ref{t: secondary invasion} reduces to verifying the existence and spectral stability conditions of Propositions \ref{p: secondary CSC existence} and \ref{p: secondary csc spectral stability}. We now state the corresponding existence and spectral stability results in the $\alpha > 1$ case, needed to prove Theorem \ref{t: primary CSC invasion}. 
	
	\subsection{Invasion of cancer-free state by CSCs --- formulation via marginal stability}\label{s: primary CSC formulation}
	
	Here we formulate results on the linear spreading speed, front existence, and spectral stability in the $\alpha > 1$ case, analogous to the results of Section \ref{s: secondary csc formulation}. 
	
	The spreading speed of cancer stem cells into the cancer-free state may be predicted from the dispersion relation of the linearization about the cancer-free state, which is given by 
	\begin{align}
		d_0(\lambda, \nu; c, \alpha, \eps) = \det \begin{pmatrix}
			D \nu^2 + c \nu I + F'(0; \alpha, \eps) - \lambda I
		\end{pmatrix}. \label{e: cf dispersion}
	\end{align}
	\begin{lemma}[Linear spreading speed of CSCs into cancer free state.]\label{l: cf linear spreading speed}
	Fix $\alpha > 1$. With $c = \cpc = 2$ and $\etapc = 1$, the dispersion relation \eqref{e: cf dispersion} satisfies the following. 
	\begin{enumerate}
		\item (Pinched double root) For $\tilde{\nu}, \lambda$ small, we have the expansion
		\begin{align}
			d_0 (\lambda, -\etapc + \tilde{\nu}; \cpc, \alpha, \eps) = d_0^{10} \lambda + d_0^{02} \tilde{\nu}^2 + \mathrm{O} (\lambda^2, \lambda \tilde{\nu}, \tilde{\nu}^3) \label{e: cf double root}
		\end{align}
		where $d_0^{10} (\alpha, \eps) > 0$ and $d_0^{02} (\alpha, \eps) < 0$. 
		\item (Minimal critical spectrum) If $d_0(i \omega, -\etapc + ik; \cpc), \alpha, \eps) = 0$ for some $\omega, k \in \R$, then $\omega = k = 0$. 
		\item (No unstable essential spectrum) There are no solutions to $d_0(\lambda, -\etapc + ik; \cpc, \alpha, \eps) = 0$ with $\Re \lambda > 0$ and $k \in \R$. 
	\end{enumerate}
	\end{lemma}
	The proof of Lemma \ref{l: cf linear spreading speed} is a direct calculation, identical to that of Lemma \ref{l: tc linear spreading speed}, so we omit the details. 
	\begin{prop}[Existence of critical primary CSC front]\label{p: primary CSC existence}
		Fix $\alpha > 1$. There exists $\eps_*(\alpha) > 0$ such that if $0 < \eps < \eps_*(\alpha)$, then \eqref{e: tw equation} with speed $c = \cpc = 2$ admits a primary CSC front solution $\Upcf (\xi; \alpha, \eps) = (\upcf(\xi; \alpha, \eps), \vpcf(\xi; \alpha, \eps))^T$, with Jordan block-type asymptotics as $\xi \to \infty$, 
		\begin{align}
			\begin{pmatrix}
				\upcf(\xi; \alpha, \eps) \\ \vpcf(\xi; \alpha, \eps)
			\end{pmatrix} = [\u_0^0 \xi + \u_0^1 + a \u_0^0] e^{-\xi} + \mathrm{O}(e^{(-1 - \eta)\xi}), \label{e: primary csc front asymptotics}
		\end{align}
		for some $\eta > 0$, $\u_0^0, \u_0^1 \in \R^2$ and $a \in \R$. In the wake, we have exponential convergence,
		\begin{align}
			\begin{pmatrix}
				\upcf(\xi) \\ \vpcf(\xi)
			\end{pmatrix} = \begin{pmatrix} 1 \\ 0 \end{pmatrix}
			+ \mathrm{O} (e^{\eta \xi}), 
		\end{align}
		as $\xi \to -\infty$, for some $\eta > 0$. 
	\end{prop}
	Let 
	\begin{align}
		\Apcf = \partial_\xi^2 + 2 \partial_\xi + F'(\Upcf(\xi; \alpha, \eps); \alpha, \eps) 
	\end{align}
	denote the linearization about a primary CSC front from Proposition \ref{p: primary CSC existence}. The essential spectrum of $\Apcf$ is again unstable due to the instability of the invaded background state $(u,v) = (0,0)$. We therefore let $\omegapc$ be a smooth positive weight function satisfying
	\begin{align*}
		\omegapc(\xi) = \begin{cases}
			1, & \xi \leq -1, \\
			e^{\xi} &\xi \geq 1,
		\end{cases}
	\end{align*}
	and define
	\begin{align}
		\Lpcf \u = \omegapc \Apcf \left( \frac{1}{\omegapc} \u \right). 
	\end{align}
	\begin{prop}[Marginal stability of secondary CSC front]\label{p: primary csc spectral stability}
		Fix $\alpha > 1$. There exists $\eps_*(\alpha) > 0$ such that if $0 < \eps < \eps_*(\alpha)$, then the essential spectrum of $\Lpcf : H^2 (\R, \C^2) \subset L^2(\R, \C^2) \to L^2(\R, \C^2)$ is marginally stable, touching the imaginary axis only at the origin via the curve $\{ -k^2 : k \in \R\}$, and otherwise contained in the left half plane. Moreover, $\Lpcf$ has no eigenvalues $\lambda$ with $\Re \lambda \geq 0$, and there is no bounded solution to $\Lpcf \u = 0$. 
	\end{prop}
	
	Together, Lemma \ref{l: cf linear spreading speed}, Proposition \ref{p: primary CSC existence}, and Proposition \ref{p: primary csc spectral stability} establish that for $\alpha > 1$ and $\eps$ sufficiently small, the primary CSC fronts in \eqref{e: eqn} satisfy \cite[Hypotheses 1-4]{AverySelectionRD}, and so Theorem \ref{t: primary CSC invasion} immediately follows.

	\section{Existence of fronts}\label{s: existence}
	We focus on the staged invasion regime, $0 < \alpha < 1$. Proofs in the TC $\alpha > 1$ are analogous, and we comment on the modifications in Section \ref{s: primary CSC}. Setting $\eps = \delta^2 \geq 0$ and $c = \csc (\alpha) = 2 \sqrt { \alpha}$, we write the traveling wave system \eqref{e: tw equation} as 
	\begin{align}
		0 &= u_{\xi \xi} + 2 u_\xi + f(u,v) \nonumber \\
		0 &= \delta^2  v_{\xi \xi} + \delta^2 c v_\xi + g(u, v; \alpha, \delta), \label{e: tw 2}
	\end{align}
	where
	\begin{align}
		f(u,v) = u(1-u-v), \quad g(u,v;\alpha, \delta^2) = (1-\delta^2) (1-u-v)u + (1-u-v)v - \alpha v. 
	\end{align}
	
	\subsection{Regularization via geometric singular perturbation theory}\label{s: gspt}
	We now reformulate \eqref{e: tw 2} as the first order system
	\begin{align}
		u_\xi &= w \nonumber \\
		w_\xi &= - \frac{1}{d} [cw + f(u,v)] \nonumber \\
		\delta v_\xi &= z \nonumber \\
		\delta z_\xi &= - \frac{1}{d} \left[ \delta c z + g(u,v; \alpha, \delta) \right]. \label{e: slow sys}
	\end{align}
	We refer to \eqref{e: slow sys} as the \emph{slow system}. For $\delta > 0$, we can rescale the spatial variable to $\zeta = \xi/\delta$, finding the \emph{fast system}
	\begin{align}
		u_\zeta &= \delta w \nonumber \\
		w_\zeta &= -  \delta [cw + f(u,v)] \nonumber \\
		v_\zeta &= z \nonumber \\
		z_\zeta &= -  \left[ \delta c z + g(u,v,;\alpha, \delta) \right]. \label{e: fast sys}
	\end{align}
	When $\delta = 0$, the fast system \eqref{e: fast sys} admits a manifold of equilibria $\mathcal{M}_0$ given by
	\begin{align}
		\mathcal{M}_0 = \left\{ (u, w, v, z) \in \R^4 : z = 0, v = v_\alpha (u) \right\}, \label{e: M0}
	\end{align}
	where $v_\alpha(u)$ is found by solving $g(u,v; \alpha, 0) = 0$, and has the explicit formula
	\begin{align}
		v_\alpha(u) = \frac{1-\alpha - 2 u}{2} + \sqrt{ u - u^2 + \frac{1}{4} (1-2u-\alpha)^2}. 
	\end{align}
	We note in particular that $v_\alpha(0) = 1-\alpha$ and $v_\alpha(1) = 0$, so this slow manifold connects to both the pure TC state $(u,v) = (0, 1-\alpha)$ and the pure CSC state $(u,v) = (1,0)$. The quadratic equation $g(u,v;\alpha, 0) = 0$ of course has another branch of solutions; we choose the one for which $v_\alpha(u) \geq 0$ for $0 < u < 1$, $0 < \alpha < 1$, so that the population of TCs remains non-negative.  
	
	From a short calculation, one readily finds that any portion of $\mathcal{M}_0$ which satisfies $u > - \frac{(1-\alpha)^2}{4 \alpha}$ is normally hyperbolic. Fenichel's persistence theorem \cite{Fenichel} implies that any compact, normally hyperbolic portion of $\mathcal{M}_0$ persists as a locally invariant slow manifold $\mathcal{M}_\delta$. Exploiting invariance of $\mathcal{M}_\delta$, we readily arrive at the following reduction.
	
	\begin{prop}[Regularization of existence problem]\label{p: existence regularization}
		Fix $M > 1$, and an integer $k \geq 1$. There exists $\bar{\delta} > 0$ such that all trajectories of \eqref{e: slow sys} with $|\delta| < \bar{\delta}$ satisfying
		\begin{align}
			- \frac{(1-\alpha)^2}{8\alpha} \leq u \leq M, \quad |w| + |v| + |z| \leq M \label{e: existence reduction restrictions}
		\end{align}
		lie in a slow manifold $\mathcal{M}_\delta$, which is normally hyperbolic and locally invariant under the flow of \eqref{e: slow sys}. The slow manifold may be written as a graph
		\begin{align}
			\mathcal{M}_\delta = \left\{ (u,w,v,z) \in \R^4 : v = \psi_v(u,w; \alpha, \delta), z = \psi_z (u,w; \alpha, \delta) \right\}
		\end{align}
		where $\psi_{v/z}$ are $C^k$ in all arguments. It follows from invariance of $\mathcal{M}_\delta$ that along all such trajectories of \eqref{e: slow sys}, $u$ and $w$ solve the reduced system
		\begin{align}
			u' &= w \nonumber \\
			w' &= - \left[ c w + f(u, \psi_v(u, w; \alpha, \delta)) \right]. \label{e: reduced sys}
		\end{align}
		Moreover, $\psi_v$ satisfies the following:
		\begin{itemize}
			\item $\psi_v(u, w; \alpha, 0) = v_\alpha(u)$;
			\item $\psi_v(0,0; \alpha, \delta) = 1- \alpha$ for all $|\delta| < \bar{\delta}$;
			\item and $\psi_v(1, 0; \alpha, \delta) = 0$ for all $|\delta| < \bar{\delta}$. 
		\end{itemize}
	\end{prop}
	Restricting in \eqref{e: existence reduction restrictions} to trajectories with $u, w$ bounded guarantees that we stay on the slow manifold, since trajectories in the hyperbolic directions are unbounded. Restricting to $u \geq - \frac{(1-\alpha)^2}{8 \alpha}$ guarantees that we stay in the normally hyperbolic portion of the slow manifold. For the rest of the paper, we fix a $k \geq 1$ as referred to in Proposition \ref{p: existence regularization}. 
	
	The fronts we construct in the proof of Proposition \ref{p: secondary CSC existence} will satisfy \eqref{e: existence reduction restrictions}, and so we may find them as solutions to the reduced system \eqref{e: reduced sys}, which we can rewrite as the scalar traveling wave equation
	\begin{align}
		u_{\xi \xi} + c u_\xi + f(u, \psi_v(u,u_\xi; \alpha, \delta)) = 0, \label{e: reduced scalar}
	\end{align}
	which now depends $C^k$-smoothly on the parameter $\delta$. Hence, we have regularized the singular perturbation and can now study the persistence of fronts near the $\delta = 0$ limit to construct fronts to prove Proposition \ref{p: secondary CSC existence}. This analysis follows essentially from \cite[Theorem 2]{AveryScheelSelection}, but we give a detailed proof for completeness. 
	
	\subsection{Far-field/core decomposition for front existence}
	Since $c = \csc = 2 \sqrt{\alpha}$ is the linear spreading speed associated to the TC state in \eqref{e: eqn} in this regime, it follows that the linearization of \eqref{e: reduced sys} at $u = w = 0$ has a Jordan block (one may also readily verify this with a short direct calculation), and hence solutions decay exponentially with an algebraically growing prefactor. We therefore aim to capture solutions to \eqref{e: reduced scalar} with the \emph{far-field/core ansatz}
	\begin{align}
		u(\xi) = \chi_-(\xi) + u_c(\xi) + \chi_+(\xi) (\xi + a) e^{-\etasc(\alpha) \xi}, \label{e: existence ansatz}
	\end{align}
	where $\chi_+$ is a smooth positive cutoff function satisfying
	\begin{align}
		\chi_+(\xi) = \begin{cases}
			1, & \xi \geq 3, \\
			0, & \xi \leq 2,
		\end{cases}
	\end{align}
	and $\chi_-(\xi) = \chi_+(-\xi)$. The core function $u_c(\xi)$ will be required to be exponentially localized, with decay rate faster than $e^{-\etasc (\alpha) \xi}$ as $\xi \to \infty$. The ansatz \eqref{e: existence ansatz} therefore captures convergence to the pure CSC state $(u, v) = (1, \psi_v(1, 0; \alpha, \delta)) = (1,0)$ in the wake of the front, and convergence to the pure TC state $(u,v) = (0, \psi_v(0,0;\alpha, \delta)) = (0,1-\alpha)$ with weak exponential decay in the leading edge. 
	
	At $\delta = 0$, \eqref{e: reduced sys} has the explicit form
	\begin{align}
		0 = u_{\xi \xi} + \csc(\alpha) u_\xi + \tilde{f}(u; \alpha), \label{e: reduced scalar delta 0}
	\end{align}
	with $\tilde{f}(u; \alpha) = f(u, v_\alpha(u))$. One may readily verify that $\tilde{f}(0; \alpha) = \tilde{f}'(1; \alpha) = 0$, and $\tilde{f}'(0; \alpha) = \alpha > 0$ while $\tilde{f}'(1; \alpha) < 0$. Furthermore, it has been shown in \cite[proof of Theorem 1]{Hillen2} that $\tilde{f}$ satisfies the Fisher-KPP condition $\tilde{f}'(u;\alpha) \leq \tilde{f}'(0;\alpha) u$ for $0 \leq u \leq 1$. It then follows from the classical work of Aronson and Weinberger \cite{aronson} that \eqref{e: reduced scalar delta 0} admits a positive, monotone front solution $\ukpp$ satisfying
	\begin{align*}
		\lim_{\xi \to -\infty} \ukpp(\xi) = 1, \quad \lim_{\xi \to \infty} \ukpp(\xi) = 0. 
	\end{align*}
	A slight extension due to Gallay \cite{Gallay,EckmannWayne} guarantees that $\ukpp$ has the form \eqref{e: existence ansatz}, with some coefficient $a = \akpp > 0$. In general, the asymptotics in the leading edge would have the form $(b \xi + a)e^{-\etasc(\alpha) x}$, but we can always guarantee $b = 1$ by translating the front in space. We can linearize the equation at $\delta = 0$ about this solution to find the Fisher-KPP linearization
	\begin{align}
		\Akpp = \partial_\xi^2 + \csc(\alpha) \partial_\xi + \tilde{f}'(\ukpp; \alpha). \label{e: Akpp def}
	\end{align}
	
	Inserting the ansatz \eqref{e: existence ansatz} into \eqref{e: reduced scalar} gives an equation which we aim to solve for the exponentially localized core correction $u_c$ and far-field parameter $a \in \R$. By the above discussion, we have a solution for $\delta = 0$ corresponding to $\ukpp$, and so we hope to continue this solution to nonzero $\delta$ via the implicit function theorem. In order to do this, we need to make our requirement of exponential localization of $u_c$ more precise, and explain how this recovers Fredholm properties of the linearization.
	
	\subsection{Exponentially weighted spaces} 
	Given rates $\eta_\pm \in \R$, we define a smooth positive weight function $\omega_{\eta_-, \eta_+}$ satisfying
	\begin{align}
		\omega_{\eta_-, \eta_+}(\xi) = \begin{cases}
			e^{\eta_- \xi}, & \xi \leq -1, \\
			e^{\eta_+ \xi}, & \xi \geq 1. 
		\end{cases}
	\end{align}
	Given additionally non-negative integers $k$ and $m$ and a field $\F \in \{ \R, \C\}$, we define an exponentially weighted Sobolev space $H^k_{\eta_-, \eta_+}(\R, \F^m)$ as follows
	\begin{align}
		H^k_{\eta_-, \eta_+} (\R, \F^m) = \left\{ u \in H^k_\mathrm{loc}(\R, \F^m) : \| \omega_{\eta_-, \eta_+} u \|_{H^k} < \infty \right\}. 
	\end{align} 
	We equip this space with the norm $\| u \|_{H^k_{\eta_-, \eta_+}} = \| \omega_{\eta_-, \eta_+} u \|_{H^k}$. 
	
	\subsection{Fredholm properties and the implicit function theorem --- proof of Proposition \ref{p: secondary CSC existence}}
	
	We now recall basic spectral and Fredholm properties of $\Akpp$ which we will use in our arguments. 
	\begin{lemma}[Spectral and Fredholm properties of KPP fronts]\label{l: kpp fredholm}
		The Fisher-KPP linearization $\Akpp$ defined in \eqref{e: Akpp def} satisfies the following properties. 
		\begin{itemize}
			\item Spectral stability: considered as an operator $\Akpp: H^2_{0, \etasc} (\R) \subset L^2_{0, \etasc}(\R) \to L^2_{0, \etasc}(\R)$, $\Akpp$ has marginally stable essential spectrum and no unstable point spectrum. 
			\item No resonance at $\lambda = 0$: there is no solution to the equation $\Akpp u = 0$ for which $\omegasc u$ is bounded. 
			\item Fix $\tilde{\eta} > 0$ sufficiently small, and set $\eta = \etasc + \tilde{\eta}$. Then $\Akpp: H^2_{0, \eta} (\R) \subset L^2_{0, \eta} (\R) \to L^2_{0, \eta} (\R)$ is Fredholm with index -1. 
		\end{itemize}
	\end{lemma}
	\begin{proof}
		The claims on essential spectrum and Fredholm properties follow from a short calculation using Palmer's theorem \cite{Palmer2}; see any of \cite{SandstedeReview, FiedlerScheel, KapitulaPromislow} for a review of Fredholm properties of linearizations about traveling waves. Stability of point spectrum follows from a Sturm-Liouville type argument; see \cite[Theorem 5.5]{aronson}. One finds the second property by using the fact that $\partial_x \ukpp(\xi) \sim \xi e^{-\etasc(\alpha) \xi}, \xi \to \infty$ solves $\Akpp u = 0$, and examining the Wronskian of this linear ODE. 
	\end{proof}
	
	Inserting the ansatz \eqref{e: existence ansatz} into \eqref{e: reduced scalar}, we arrive at the equation
	\begin{align*}
		F_\alpha (u_c, a; \delta) = 0, 
	\end{align*}
	where
	\begin{multline}
		F_\alpha (u_c, a; \delta) = \partial_\xi^2 (\chi_- + u_c + \chi_+ \phi) + \csc(\alpha) \partial_\xi (\chi_- + u_c + \chi_+ \phi) \\ + f(\chi_- + u_c + \chi_+ \phi, \psi_v(\chi_- + u_c + \chi_+ \phi, \chi_-' + u_c' + \partial_\xi (\chi_+ \phi); \alpha, \delta)), \label{e: F existence def}
	\end{multline}
	with $\phi(\xi; a, \alpha) = (\xi + a) e^{-\etasc(\alpha) \xi}$. 
	\begin{lemma}
		Fix $0 < \alpha < 1$ and $\tilde{\eta} > 0$ small, and set $\eta = \etasc(\alpha) + \tilde{\eta}$. There exists $\delta_0 > 0$ such that the function
		\begin{align*}
			F_\alpha: H^2_{0, \eta} (\R, \R) \times \R \times (-\delta_0, \delta_0) \to L^2_{0, \eta} (\R, \R)
		\end{align*}
		is well defined and $C^k$ in all arguments. 
	\end{lemma}
	\begin{proof}
		The main task to prove is that $F_\alpha$ preserves exponential localization. It follows from Lemma \ref{l: cf linear spreading speed} that the function $(\xi + a)e^{-\etasc(\alpha) \xi}$ explicitly solves the linearized equation about $u \equiv 0$. All residual terms from inserting the ansatz \eqref{e: existence ansatz} therefore are either at least as localized as $u_c$, or are generated by quadratic powers of the far-field term, and hence are $\mathrm{O}(\xi^2 e^{-2 \etasc(\alpha) \xi})$, and so belong to $L^2_{0, \eta} (\R, \R)$, as desired. 
	\end{proof}
	
	Note that $F_\alpha(u_c^\mathrm{kpp}, \akpp; 0) = 0$, where 
	\begin{align*}
		u_c^\mathrm{kpp}(\xi) = \ukpp(\xi) - \chi_-(\xi) - \chi_+(\xi)(\xi + \akpp)e^{-\etasc(\alpha) \xi} \in H^2_{0, \eta}. 
	\end{align*}
	The linearization about this solution is $D_{u_c} F(u_c^\mathrm{kpp}, \akpp; 0) = \Akpp$, which is Fredholm with index -1 when considered as an operator $\Akpp : H^2_{0, \eta} (\R) \to L^2_{0, \eta} (\R)$, by Lemma \ref{l: kpp fredholm}. 
	
	Adding the additional parameter $a$ by considering the joint linearization $D_{u_c, a} F(\uckpp, \akpp; 0)$ then increases the Fredholm index by 1, a fact sometimes referred to as the Fredholm bordering lemma. Hence, this joint linearization is Fredholm with index 0, and so is invertible if and only if it has trivial kernel. 
	\begin{lemma}\label{l: existence joint invertibility}
		 Fix $0 < \alpha < 1$ and $\tilde{\eta} > 0$ sufficiently small, and set $\eta = \etasc(\alpha) + \tilde{\eta}$. The joint linearization $D_{u_c, a} F(\uckpp, \akpp; 0) : H^2_{0, \eta} (\R) \to L^2_{0, \eta} (\R)$ is invertible. 
	\end{lemma}
	\begin{proof}
		We only need to verify that the kernel is trivial. From a short calculation, we find that an element $(\tilde{u}, \tilde{a}) \in H^2_{0, \eta} (\R) \times \R$ of the kernel satisfies 
		\begin{align*}
			\Akpp (\tilde{u}+ \chi_+ \tilde{a} e^{-\etasc(\alpha) \cdot} ) = 0. 
		\end{align*}
		If one of $\tilde{u}$ or $\tilde{a}$ were nonzero, then we would have a bounded solution to $\Akpp u = 0$, which is excluded by Lemma \ref{l: kpp fredholm}. Hence, the kernel is trivial, as desired. 
	\end{proof}
	
	\begin{proof}[Proof of Proposition \ref{p: secondary CSC existence}]
		By Lemma \ref{l: existence joint invertibility}, we can use the implicit function theorem to construct solutions $\uscf(\xi; \alpha, \delta)$ to the reduced equation \eqref{e: reduced scalar} for $\delta \ll 1$, with the form \eqref{e: existence ansatz}. By Proposition \ref{p: existence regularization}, setting $\vscf(\xi; \alpha, \delta) = \psi_v (\uscf(\xi; \alpha, \delta), \partial_\xi \uscf(\xi; \alpha, \delta); \alpha, \delta)$, we find the desired front solutions to the full system \eqref{e: slow sys}, for which one can readily verify the asymptotics \eqref{e: secondary csc front asymptotics}. 
		
		Positivity can be established using a simple maximum principle argument, as follows. At $\delta = 0$, the front $\ukpp$ satisfies $0 < \ukpp < 1$. Since $\uscf$ depends continuously in $\delta$ and the asymptotics \eqref{e: existence ansatz} imply positive in the leading edge and the wake, if $\uscf$ were not strictly positive for some $\delta_1 > 0$, then there would have to be $\xi_0 \in \R$ and $\delta_2> 0$ such that $\partial_\xi^2 \uscf (\xi_0) > 0$, $\partial_\xi \uscf (\xi_0) = u(\xi_0) = 0$. A simple inspection of the first equation of \eqref{e: tw 2} shows that this cannot happen, so $\uscf$ must remain positive as $\delta$ increases from 0. A similar argument establishes positivity for $\vscf$. At $\delta = 0$, one can use a maximum principle argument to show that $\uscf + \vscf \geq \frac{1-\alpha}{2}$. By continuity, we then obtain $\uscf + \vscf > \frac{3(1-\alpha)}{4}$ for $\delta$ sufficiently small, as desired. 
	\end{proof}
	
	\subsection{Existence of primary CSC fronts --- proof of Proposition \ref{p: primary CSC existence}}\label{s: primary CSC}
	The proof of Proposition \ref{p: primary CSC existence}, establishing existence of pulled primary CSC fronts in the TC extinction regime $\alpha > 1$, is almost identical to that of Proposition \ref{p: primary CSC existence}. The main difference is that for $\alpha > 1$, the physically relevant branch of the slow manifold is now
	\begin{align*}
		\mathcal{M}_0^- = \left\{ (u, w, v, z) \in \R^4 : z = 0, v = v_\alpha^-(u) \right\},
	\end{align*}
	where 
	\begin{align*}
		v_\alpha^- (u) = \frac{1-\alpha-2u}{2} - \sqrt{u-u^2 + \frac{1}{4} (1-2u-\alpha)},
	\end{align*}
	satisfies $v_\alpha^-(0) = v_\alpha^-(1) = 0$, and $v_\alpha^-(u) \geq 0$ for $0 \leq u \leq 1$, $\alpha > 1$. The slow manifold $\mathcal{M}_0^-$ is still normally hyperbolic for $u > -\frac{(1-\alpha)^2}{4 \alpha}$, and the analysis of \cite{Hillen2} again establishes that the resulting reduced equation on the slow manifold is of Fisher-KPP type for $\delta = 0$. The remaining details are completely analogous to the proof of Proposition \ref{p: secondary CSC existence}. 
	
	\section{Spectral stability of CSC fronts }\label{s: spectral}
	
	We again focus on the staged invasion regime, $0 < \alpha < 1$. The proof of Proposition \ref{p: primary csc spectral stability} in the CSC dominated regime is completely analogous. We again set $\eps = \delta^2 \geq 0$ and $c = \csc(\alpha) = 2 \sqrt{\alpha}$. 
	
	In Section \ref{s: framework}, we divided the second equation in \eqref{e: eqn} by $\eps$ to obtain the form \eqref{e: sys divided by eps}, so that we could directly apply the results of \cite{AverySelectionRD}. In actually proving Proposition \ref{p: secondary csc spectral stability}, it will be more convenient to work with the linearization of \eqref{e: eqn} about the front solution from Proposition \ref{p: secondary CSC existence}, which leads to the eigenvalue problem
	\begin{align*}
		(\Bsc(\delta) - K_\delta \lambda) \begin{pmatrix} \tilde{u} \\ \tilde{v} \end{pmatrix} = 0, \label{e: stability problem}
	\end{align*}
	where
	\begin{align}
		\Bsc(\delta) = \begin{pmatrix}
			\partial_\xi^2 + \csc(\alpha) \partial_\xi + \partial_u f(u, v) & \partial_v f(u, v) \\
			\partial_u g(u, v; \alpha, \delta) & \delta^2 \partial_\xi^2 + \delta^2 \csc (\alpha) \partial_\xi + \partial_v g(u, v; \alpha, \delta)
		\end{pmatrix},
	\end{align}
	with $(u,v) = (\uscf, \vscf)$, and 
	\begin{align}
		\quad K_\delta = \begin{pmatrix}
			1 & 0 \\
			0 & \delta^2
		\end{pmatrix}
	\end{align}
	One readily obtains the following equivalence to the formulation of Section \ref{s: framework}. 
	\begin{lemma}
		Spectral properties of $\Ascf$ and $\Bsc$ are equivalent. That is, $\lambda$ is in the essential spectrum of $\Ascf$ if and only if $\Bsc(\delta) - K_\delta \lambda$ is not Fredholm with index 0, and $\lambda$ is in the point spectrum of $\Ascf$ if and only if $\Bsc(\delta) - K_\delta \lambda$ has a nontrivial kernel. 
	\end{lemma}
		
	Throughout this section, we let $c = \csc(\alpha)$, and we rewrite \eqref{e: stability problem} as the first order system
	\begin{align}
		\partial_\xi \tilde{u} &= \tilde{w} \nonumber \\
		\partial_\xi \tilde{w} &= -  [c \tilde{w} + f_u \tilde{u}+ f_v \tilde{v}- \lambda \tilde{u}] \nonumber \\
		\delta \partial_\xi \tilde{v} &= \tilde{z} \nonumber \\
		\delta \partial_\xi \tilde{z} &= -  [\delta c \tilde{z} + g_u \tilde{u} + g_v \tilde{v} - \delta^2 \lambda \tilde{v}], \label{e: large lambda lin 1}
	\end{align}
	where $f_u = \partial_u f(u,v)$, with corresponding notation for the other derivatives of the nonlinearity. 
	
	We divide the spectral stability analysis into two regimes: one where $|\lambda|$ is large, independently of $\delta$, and one where $\lambda$ is bounded. 
	
	\begin{prop}[Stability for $|\lambda|$ large]\label{p: large lambda stability}
		Fix $0 < \alpha < 1$, let $c = \csc(\alpha)$, and let $(u,v)$ denote a secondary CSC front constructed in Proposition \ref{p: secondary CSC existence}. There exist constants $\Lambda_0, \delta_0 > 0$ and $\frac{\pi}{2} < \theta_0 < \pi$ such that the equation $(\Bsc(\delta) - \lambda K_\delta) \tilde{\u} = 0$ has no bounded solutions provided $|\delta| < \delta_0$ and $\lambda$ satisfies
		\begin{align}
			\lambda \in \Omega_0 := \{ \lambda \in \C: |\lambda| \geq \Lambda_0, |\Arg \lambda | \leq \theta_0 \}. 
		\end{align}
	\end{prop}
	We prove Proposition \ref{p: large lambda stability} in Section \ref{s: large lambda}. To handle the regime $|\lambda| \leq |\Lambda_0|$, we will use a slow manifold reduction to regularize the singular perturbation in \eqref{e: large lambda lin 1}. Since the linearized system \eqref{e: large lambda lin 1} is non-autonomous, depending on $\xi$ through $u$ and $v$, we first couple \eqref{e: large lambda lin 1} to the existence problem for the secondary CSC front, obtaining the autonomous 8-dimensional system
	\begin{align}
			u_\xi &= w \nonumber, &\partial_\xi \tilde{u} &= \tilde{w}  \\
			w_\xi &= - \frac{1}{d} [cw + f(u,v)] \nonumber, &\partial_{\xi} \tilde{w} &= - [c\tilde{w}+ f_u \tilde{u} + f_v \tilde{v} - \lambda \tilde{u}], \\
			\delta v_\xi &= z \nonumber, &\delta \partial_\xi \tilde{v} &= \tilde{z},  \\
			\delta z_\xi &= - \frac{1}{d} \left[ \delta c z + g(u,v; \alpha, \delta) \right], &\delta \partial_\xi \tilde{z} &= -[\delta c \tilde{z}+ g_u \tilde{u} + g_v \tilde{v}- \delta^2 \lambda \tilde{v}] \label{e: slow existence stability coupled}
	\end{align} 
	Rescaling to find the corresponding fast system and setting $\delta = 0$, we find the slow manifold
	\begin{align}
		\tilde{\mathcal{M}}_0 = \left\{ (u,w, v, z, \tilde{u}, \tilde{w}, \tilde{v}, \tilde{z}) \in \C^8 : (u,v,w,z) \in \mathcal{M}_0, \tilde{z} = 0, \tilde{v} = - \frac{g_u (u,v;\alpha, 0)}{g_u(u,v;\alpha, 0)} \tilde{u} \right\},
	\end{align}
	where $\mathcal{M}_0$ is given by \eqref{e: M0}. We establish the persistence of this slow manifold to $\delta \neq 0$ in the following proposition. 
	
	\begin{prop}[Reduction for $|\lambda|$ bounded]\label{p: eigenvalue reduction}
		Fix $0 < \alpha < 1$, $\Lambda > 0$ and $M > 1$. Then there exists $\overline{\delta} > 0$ so that the following holds. All bounded trajectories of \eqref{e: slow existence stability coupled} with $|\delta| < \overline{\delta}$ and $|\lambda| < \Lambda$ satisfying
		\begin{align}
			-\frac{(1-\alpha)^2}{8 \alpha} \leq u \leq M, \quad |w| + |v| + |z| \leq M. 
		\end{align}
		lie in a slow manifold which is normally hyperbolic and locally invariant under the flow of \eqref{e: slow existence stability coupled} and may be written as a graph
		\begin{align}
			v &= \psi_v (u, w; \alpha, \delta), \nonumber \\
			z &= \psi_z (u,w; \alpha, \delta), \nonumber \\
			\tilde{v}&= \tilde{\psi}_v (u,w; \alpha, \lambda, \delta) \cdot (\tilde{u}, \tilde{w})^T, \nonumber \\
			\tilde{z}&= \tilde{\psi}_z (u,w; \alpha, \lambda, \delta) \cdot (\tilde{u}, \tilde{w})^T, \label{e: eigenvalue slow manifold}
		\end{align}
		where $\psi_{v/z}$ are given in Proposition \ref{p: existence regularization}, and $\tilde{\psi}_{v/z} (u,w; \alpha, \lambda, \delta) : \C^2 \to \C$ are linear operators which are $C^k$ in all arguments, with expansions
		\begin{align}
			\tilde{\psi}_v (u, w; \alpha, \lambda, \delta) &= \begin{pmatrix} \tilde{\psi}_v^1 (u, w; \alpha, \lambda, \delta) & \tilde{\psi}_v^2 (u, w; \alpha, \lambda, \delta) \end{pmatrix} =  \begin{pmatrix}
				- \frac{g_u(u, \psi_v(u, w; \alpha, 0); \alpha, 0)}{g_v(u, \psi_v(u, w; \alpha, 0); \alpha, 0)} & 0
			\end{pmatrix} + \mathrm{O}(\delta), \nonumber \\
			\tilde{\psi}_z (u,w; \alpha, \lambda, \delta) &= \mathrm{O}(\delta). \label{e: eigenvalue slow manifold expansions}
		\end{align}
		As a result, in the parameter regime considered here, we have a bounded solution to the eigenvalue problem \eqref{e: large lambda lin 1} if and only if we have a bounded solution to the reduced problem
		\begin{align}
			\partial_\xi \tilde{u} &= \tilde{w}, \nonumber \\
			\partial_\xi \tilde{w} &= - [c \tilde{w} + f_u (u, \psi_v (u, w; \alpha, \delta)) \tilde{u} + f_v(u, \psi_v(u, w; \alpha, \delta); \alpha, \delta) \tilde{v} - \lambda \tilde{u}],
		\end{align}
		where $\tilde{v}$ is given by \eqref{e: eigenvalue slow manifold}. 
	\end{prop}
	\begin{proof}
		This follows from a modification of the proof of Fenichel's persistence theorem \cite{Fenichel}. First, Fenichel's theorem usually applies to compact manifolds, but this compactness is needed only to guarantee uniformity of the hyperbolicity constants, and here since the second set of equations is linear, one readily finds that hyperbolicity is uniform after only restricting the front variables $(u,v,w,z)$ to a compact set. Second, the coupled system \eqref{e: slow existence stability coupled} has a special structure, namely, the coupling is triangular, with the first four equations not depending on the last four variables, and the equations for the last four variables are linear in $(\tilde{u}, \tilde{w}, \tilde{v}, \tilde{z})$. We assert in \eqref{e: slow existence stability coupled} that this property is preserved in the slow manifold reduction, so that the $v$ and $z$ components of the slow manifold depend only on the existence variables, and the $\tilde{v}, \tilde{z}$ components are linear in $(\tilde{u}, \tilde{w})$. The proof of Fenichel's persistence theorem \cite{Fenichel, Fenichel2} (see also \cite[Chapter 3]{Wiggins}) constructs the slow manifold as the fixed point of a map called the graph transform. Inspecting the details, one finds that the graph transform preserves the structure in \eqref{e: eigenvalue slow manifold}, so that by initializing with a map with this structure, one finds this structure for the resulting limit, as desired. The expansion \eqref{e: eigenvalue slow manifold expansions} can then be computed as usual using invariance of the slow manifold. 
	\end{proof}
	
	\subsection{Exponential dichotomies}
	To prove Proposition \ref{p: large lambda stability}, we will use the theory of \emph{exponential dichotomies}, so we first recall some basic definitions and results. For further background on exponential dichotomies, see for instance \cite{SandstedeReview, Coppel}. 
	\begin{defi}
		Consider an equation
		\begin{align}
			u_\xi = A (\xi) u, \label{e: example ODE}
		\end{align}
		with $u \in \C^n, A \in \C^{n \times n}$. Let $\Phi(\xi, \zeta)$ denote the flow operator for this ODE, which maps an initial condition at time $\zeta$ to the solution at time $\xi$. We say \eqref{e: example ODE} has an \emph{exponential dichotomy on} $\R$ if there exist constants $K > 0$ and $\kappa^s < 0 < \kappa^u$ together with a family of projections $P(\xi)$, continuous in $\xi \in \R$, such that the following is true for all $\xi, \zeta \in \R$. 
		\begin{itemize}
			\item Setting $\Phi^s(\xi, \zeta) = \Phi(\xi, \zeta) P(\zeta)$, we have
			\begin{align}
				|\Phi^s(\xi, \zeta)| \leq K e^{\kappa^s (\xi - \zeta)}
			\end{align}
			for all $\xi, \zeta \in \R$ with $\xi \geq \zeta$. 
			\item Setting $\Phi^u(\xi, \zeta) = \Phi(\xi, \zeta) (I - P(\zeta))$, we have
			\begin{align}
				|\Phi^u (\xi, \zeta) | \leq K e^{\kappa^u (\xi - \zeta)} 
			\end{align}
			for all $\xi, \zeta \in \R$ with $\xi \leq \zeta$. 
			\item The projections $P(\xi)$ commute with the evolution, $\Phi(\xi,\zeta) P(\zeta) = P(\xi) \Phi(\xi, \zeta)$. 
		\end{itemize}
	\end{defi}
	To summarize, the ODE \eqref{e: example ODE} has an exponential dichotomy on $\R$ if for each fixed $\xi$, we can decompose the phase space into two complementary subspaces, one which corresponds to initial conditions which decay exponentially in forward (spatial) time and one which corresponds to initial conditions which decay exponentially in backward time. One can then see that if \eqref{e: example ODE} has an exponential dichotomy, then it admits no solutions which are bounded on $\R$. Hence, if an equation $u_\xi = A(\xi, \lambda) u$ corresponds to the spectral problem for a traveling wave and has an exponential dichotomy on $\R$ for some $\lambda \in \C$, then this $\lambda$ cannot be an eigenvalue of the traveling wave linearization. Exponential dichotomies are also robust to small perturbations of the equation; see \cite[Chapter 4]{Coppel}. Finally, if the coefficient matrix $A(\xi)$ is slowly varying, then the existence of exponential dichotomies follows from hyperbolicity of the frozen coefficient system at each $\xi$; see \cite[Lemma 2.3]{Sakamoto}. 
	
	\subsection{Stability for large $|\lambda|$ --- proof of Proposition \ref{p: large lambda stability}}\label{s: large lambda}
	To study the large $\lambda$ limit, we set $\lambda = \frac{1}{\gamma^2}$ and aim to perturb from $\gamma = 0$. Defining the rescaled (spatial) time $y = \xi/\delta$, we find the \emph{fast system}
	\begin{align}
		\partial_y \tilde{u} &= \delta \tilde{w} \nonumber \\
		\partial_y \tilde{w} &= - \delta\left[c \tilde{w} + f_u \tilde{u}+ f_v \tilde{v}- \frac{1}{\gamma^2} \tilde{u}\right] \nonumber \\
		\partial_y \tilde{v} &= \tilde{z} \nonumber \\
		\partial_y \tilde{z} &= -  \left[\delta c \tilde{z} + g_u \tilde{u} + g_v \tilde{v} - \frac{\delta^2}{\gamma^2} \tilde{v} \right],
	\end{align}
	equivalent to \eqref{e: large lambda lin 1} for $\delta > 0$. To improve the singularity at $\gamma = 0$, we again rescale time to $\zeta = y/|\gamma|$ and define $\hat{w} = |\gamma| \tilde{w}$, finding the equivalent system
	\begin{align}
		\partial_\zeta \tilde{u} &= \delta \hat{w} \nonumber \\
		\partial_\zeta \hat{w} &=  \frac{|\gamma|^2}{\gamma^2} \delta \tilde{u} - \delta |\gamma|^2 \left[ c \frac{\hat{w}}{|\gamma|} + f_u \tilde{u}+ f_v \tilde{v} \right] \nonumber \\
		\partial_\zeta \tilde{v} &= |\gamma| \tilde{z} \nonumber \\
		\partial_\zeta \tilde{z} &=  \frac{\delta^2}{|\gamma|} \frac{|\gamma|^2}{\gamma^2} \tilde{v} - |\gamma|\left[ \delta c \tilde{z} + g_u \tilde{u} + g_v \tilde{v} \right] \label{e: large lambda finally rescaled}
	\end{align}
	Our goal is to show that \eqref{e: large lambda finally rescaled} admits no bounded solutions for $|\delta|, |\gamma|$ small, with the argument of $\gamma$ restricted to an appropriate sector. We proceed by dividing the first quadrant of the $(\delta, |\gamma|)$ plane into two distinct regions: the region covered by the lines $\delta = \delta_1 |\gamma|, \delta_1 \geq 0$, and the region covered by the lines $|\gamma| = \gamma_2 \delta, \gamma_2 > 0$. 
	
	\subsubsection{$\delta = \delta_1 |\gamma|$}
	First, we look along lines $\delta = \delta_1 |\gamma|$ in the $(\delta, |\gamma|)$ plane, with $\delta_1 > 0$. Along any such line, the system \eqref{e: large lambda finally rescaled} has an Euler multiplier $|\gamma|$, and so rescaling time to remove this multiplier we obtain the equivalent system
	\begin{align}
		\partial_\zeta \tilde{u} &= \delta_1 \hat{w} \nonumber \\
		\partial_\zeta \hat{w} &=  \frac{|\gamma|^2}{\gamma^2} \delta_1 \tilde{u} - \delta_1 |\gamma|^2 \left[c\frac{\hat{w}}{|\gamma|} + f_u \tilde{u} + f_v \tilde{v}\right] \nonumber \\
		\partial_\zeta \tilde{v} &= \tilde{z} \nonumber \\
		\partial_\zeta \tilde{z} &= \delta_1^2 \frac{|\gamma|^2}{\gamma^2} \tilde{v} -  [\delta_1 |\gamma| c \tilde{z} + g_u \tilde{u} + g_v \tilde{v}]. 
	\end{align}
	Writing $\tilde{U} = (\tilde{u}, \hat{w}, \tilde{v}, \tilde{z})^T$, this system has the form 
	\begin{align}
	\partial_\zeta \tilde{U} = A_0 (u, v, \delta_1, \gamma) \tilde{U} + |\gamma| A_1(u,v, \delta_1, \gamma) \tilde{U}, \label{e: large lambda first scaling}
	\end{align}
	where
	\begin{align*}
		A_0 (u, v, \delta_1, \gamma) = \begin{pmatrix}
			0 & \delta_1 & 0 & 0 \\
			\delta_1 \frac{|\gamma|^2}{\gamma^2} & 0 & 0 & 0 \\
			0 & 0 & 0 & 1 \\
			-g_u (u,v; \alpha, 0) & 0 & \delta_1^2 \frac{|\gamma|^2}{\gamma^2} - g_v (u,v; \alpha, 0) & 0 
		\end{pmatrix},
	\end{align*}
	and 
	\begin{align*}
		A_1(u,v, \delta_1, \gamma) = \begin{pmatrix}
			0 & 0 & 0 & 0 \\
			- \delta_1 |\gamma| f_u & - \delta_1 c & - \delta_1 |\gamma| f_v & 0 \\
			0 & 0 & 0 & 0 \\
			0 & 0 & 0 & - \delta_1 c 
		\end{pmatrix}.
	\end{align*}
	We would like  to show that \eqref{e: large lambda first scaling} admits no bounded solutions by constructing exponential dichotomies. Since the coefficients $(u,v)$ are slowly varying, on the timescale $\delta_1 |\gamma|^2$, we can establish exponential dichotomies by verifying that the coefficient matrix is hyperbolic for each fixed $u,v$ along the front \cite[Lemma 2.3]{Sakamoto}. However, $A_0$ loses hyperbolicity at $\delta_1 = 0$, so we treat the regime $\delta_1 \gtrsim 0$ separately by reducing to a slow manifold as in Proposition \ref{p: eigenvalue reduction}. 
	
\begin{lemma}[Exponential dichotomies for $\delta_1$ away from 0]\label{l: large lambda first scaling}
	Fix $0 < \alpha < 1$, and let $\delta_1 > 0$. Assume $u(\delta_1 |\gamma|^2 \zeta)$ and $v(\delta_1 |\gamma|^2 \zeta)$ are bounded and satisfy $\frac{3(1-\alpha)}{4} \leq u+v \leq M$. Then there exists $\gamma_*(\alpha, \delta_1) > 0$, continuous in both arguments, such that the system \eqref{e: large lambda first scaling} admits no bounded solutions provided $|\gamma| < \gamma_*(\alpha, \delta_1)$ and $|\mathrm{Arg} \, \gamma| \leq \frac{3 \pi}{8}$. 
\end{lemma}
\begin{proof}
	We set $\frac{|\gamma|^2}{\gamma^2} = e^{i \theta}$. The eigenvalues of $A_0(u,v, \delta_1, \gamma)$ are 
	\begin{align*}
		\nu_1^\pm = \pm e^{i \frac{\theta}{2}} \delta_1, \quad \nu_2^\pm = \pm \sqrt{\delta_1^2 e^{i \theta} -  g_v (u, v; \alpha, 0)} = \pm \sqrt{\delta_1^2 e^{i \theta} -  (1-\alpha -2 (u+v))}. 
	\end{align*}
	The first pair $\nu_1^\pm$ are clearly hyperbolic if $|\Arg \theta| \leq \frac{3 \pi}{4}$. For the second pair, suppose one of $\nu_2^\pm = ik$ is imaginary. Then we have
	\begin{align*}
		-k^2 = \delta_1^2 e^{i \theta} -  (1- \alpha - 2(u+v)). 
	\end{align*}
	Rearranging, we find
	\begin{align*}
		e^{i \theta} = \frac{1}{\delta_1^2} \left[ -k^2 + [1-\alpha - 2(u+v)]  \right].
	\end{align*}
	Since $u$ and $v$ are real and satisfy $(u+v) \geq \frac{3(1-\alpha)}{4}$, the right hand side is a real number bounded above by $\frac{1}{\delta_1^2} [-k^2 - \frac{(1-\alpha)}{2}]$, and in particular is negative. Hence, the only possibility is $\theta = \pi$, and so it follows that $A_0$ is hyperbolic in the desired region, with eigenvalues uniformly bounded away from the imaginary axis. Since the non-constant coefficients in the equation $\partial_\zeta \tilde{U} = A_0 (u,v, \delta_1, \gamma) \tilde{U}$ are slowly varying for $ \delta_1 |\gamma|^2$ small, (since they are functions of the slowly varying variables $u(\delta_1 |\gamma|^2 \zeta), v(\delta_1 |\gamma|^2 \zeta)$) it follows from this hyperbolicity together with \cite[Lemma 2.3]{Sakamoto} that the system $\partial_\zeta \tilde{U} = A_0 \tilde{U}$ admits exponential dichotomies, provided $|\gamma|$ is small. Hence, by roughness of exponential dichotomies \cite{Coppel}, there exists $\gamma_*(\alpha, \delta_1)$ such that the system \eqref{e: large lambda first scaling} admits exponential dichotomies provided $|\gamma| \leq \gamma_*(\alpha, \delta_1)$ and $|\Arg \theta| \leq \frac{3 \pi}{4}$. Exponential dichotomies rule out the existence of bounded solutions to this equation, as desired. 
\end{proof}

\begin{lemma}[No bounded solutions for $\delta_1 \sim 0$]\label{l: large lambda first rescaling small slope}
	Fix $0 < \alpha < 1$. Assume that $(u(\xi), v(\xi))$ is a bounded solution to \eqref{e: reduced sys}. Then there exists a constant $\delta_1^* >0$ and $\gamma_1^*(\delta_1^*)$ depending on $\delta_1^*$ but not on $|\delta_1| < \delta_1^*$ such that the system \eqref{e: large lambda first scaling} has no bounded solutions provided $|\delta_1| < \delta_1^*$ and $|\gamma| \leq \gamma_1^*(\delta_1^*)$ with $|\Arg \gamma| \leq \frac{3 \pi}{8}$. 
\end{lemma}
\begin{proof}
	Arguing as in the proof of Proposition \ref{p: eigenvalue reduction} (that is, coupling to the corresponding rescaled existence problem and applying a slow manifold reduction), we find that for $\delta_1$ small, all bounded solutions to \eqref{e: large lambda first scaling} lie on a slow manifold with $\tilde{v} = \left(- \frac{g_u}{g_v} + \mathrm{O}(\delta) \right)\tilde{u}$. Hence, we have a bounded solution to \eqref{e: large lambda first scaling} for $\delta_1$ small if and only if we have a bounded solution to the corresponding reduced problem
	\begin{align*}
		\partial_\zeta \tilde{u} &= \hat{w} \\
		\partial_\zeta \hat{w} &= e^{i \theta} \tilde{u} - |\gamma| [c w + |\gamma| f_u \tilde{u} + |\gamma| f_v \tilde{v} ],
	\end{align*}
	where $e^{i \theta} = \frac{|\gamma|^2}{\gamma^2}$. For $\gamma = 0$, we find this system has exponential dichotomies for $|\Arg \theta | \leq \frac{3 \pi}{4}$. Since the leading order system is independent of $\delta_1$, these dichotomies persist to $|\gamma|$ small by roughness \cite{Coppel}, with bound independent of $\delta_1$, as desired. 
\end{proof}

\subsubsection{$|\gamma| = \gamma_2 \delta$}
	In this scaling, the system \eqref{e: large lambda finally rescaled} admits $\delta$ as an Euler multiplier, so after rescaling time to eliminate this multiplier and also setting $|\gamma|^2/\gamma^2 = e^{i \theta}$, we obtain the system
	\begin{align}
		\partial_\zeta \tilde{u} &= \hat{w} \nonumber, \\
		 \partial_\zeta \hat{w} &= e^{i \theta} \tilde{u} - \delta \gamma_2 [ c \hat{w}+ \gamma_2 \delta f_u \tilde{u} + \gamma_2 \delta f_v \tilde{v}] \nonumber, \\
		 \partial_\zeta \tilde{v} &= \gamma_2 \tilde{z} \nonumber, \\
		 \partial_\zeta \tilde{z} &= \frac{e^{i\theta}}{\gamma_2} \tilde{v} - \gamma_2 [\delta c \tilde{z} + g_u \tilde{u}+ g_v \tilde{v}]. \label{e: large lambda second scaling}
	\end{align}
	Now, we define $\hat{z} = \gamma_2 \tilde{z}$ to find the equivalent system
	\begin{align*}
	\partial_\zeta \tilde{u} &= \hat{w} \nonumber, \\
\partial_\zeta \hat{w} &= e^{i \theta} \tilde{u} - \delta \gamma_2 [ c \hat{w}+ \gamma_2 \delta f_u \tilde{u} + \gamma_2 \delta f_v \tilde{v}] \nonumber, \\
\partial_\zeta \tilde{v} &= \hat{z} \nonumber, \\
\partial_\zeta \hat{z} &= e^{i\theta} \tilde{v} - \gamma_2 [\delta c \tilde{z} + \gamma_2 g_u \tilde{u}+ \gamma_2 |g_v \tilde{v}].
	\end{align*}
	The leading order part at $\gamma_2 = 0$ is hyperbolic if $|\Arg \theta| < \frac{3 \pi}{4}$, so we readily obtain the following result by roughness of exponential dichotomies \cite{Coppel}. 
\begin{lemma}\label{l: large lambda second scaling}
	Fix $0 < \alpha < 1$ and assume that $(u, u_\xi, v, v_\xi)$ is a bounded solution to \eqref{e: slow sys}. Then there exists a constant $\gamma_2^* > 0$ such that \eqref{e: large lambda second scaling} admits no bounded solutions provided $|\delta| <1, |\gamma_1| < \gamma_2^*$, and $|\Arg \theta| \leq \frac{3 \pi}{4}$. 
\end{lemma}

\subsubsection{Combining the different charts --- proof of Proposition \ref{p: large lambda stability}}
\begin{proof}[Proof of Proposition \ref{p: large lambda stability}]
	Fix $0 < \alpha < 1$. First, by Lemma \ref{l: large lambda second scaling} there exist $\gamma_2^*, \delta_2^* > 0$ such that if $|\gamma| = \gamma_2 \delta$ with $|\gamma_2| < \gamma_2^*$, $|\delta| < \delta_2^*$, and $|\Arg \gamma | \leq \frac{3 \pi}{8}$, then \eqref{e: large lambda finally rescaled} admits no bounded solutions. Also by Lemma \ref{l: large lambda first rescaling small slope}, there exist $\delta_1^*, \gamma_1^*(\delta_1^*) > 0$ such that \eqref{e: large lambda finally rescaled} has no bounded solutions provided $\delta = \delta_1 |\gamma|$, with $|\delta_1| < \delta_1^*$, $|\gamma| < \gamma_1^*(\delta_1^*)$, and $|\Arg \gamma| \leq \frac{3 \pi}{8}$. We can then apply Lemma \ref{l: large lambda first scaling} for each $\delta_1$ in the compact interval $\left[ \frac{\delta_1^*}{2}, \frac{2}{\gamma_2^*}\right]$ such that \eqref{e: large lambda finally rescaled} has no bounded solutions provided $|\gamma| < \gamma_*(\alpha, \delta_1)$, $|\Arg \gamma| \leq \frac{3 \pi}{8}$. Since this $\gamma_*$ depends continuously on $\delta_1$ and is positive everywhere on the compact interval $\left[ \frac{\delta_1^*}{2}, \frac{2}{\gamma_2^*}\right]$, there exists $\gamma_0^* (\alpha, \delta_1^*, \gamma_2^*) > 0$ such that \eqref{e: large lambda finally rescaled} has no bounded solutions for all $\delta_1 \in \left[ \frac{\delta_1^*}{2}, \frac{2}{\gamma_2^*}\right]$ provided $|\gamma| \leq \gamma_0^*(\alpha, \delta_1^*, \gamma_2^*)$ with $|\Arg \gamma| \leq \frac{3 \pi}{8}$. Setting $\gamma_0 (\alpha, \delta_1^*, \gamma_2^*) = \min \left\{ \gamma_2^*, \gamma_1^*(\delta_1^*), \gamma_0^*(\alpha, \delta_1^*, \gamma_2^*) \right\}$ and $\delta_0 (\alpha, \delta_1^*, \gamma_2^*) = \min \{ 1, \frac{2}{\gamma_2^*} \gamma_0^*(\alpha, \delta_1^*, \gamma_2^*) \}$, we find that \eqref{e: large lambda finally rescaled} has no bounded solutions with $\delta = \delta_1 |\gamma|$ for all $\delta_1 \in (0, \infty)$ provided $|\gamma| < \gamma_0(\alpha, \delta_1^*, \gamma_2^*)$, $|\delta| \leq \delta_0(\alpha, \delta_1^*, \gamma_2^*)$  and $|\Arg \gamma| \leq \frac{3 \pi}{8}$. Since $\delta_1^*$ and $\gamma_2^*$ are independent of one another, this completes the proof of Proposition \ref{p: large lambda stability} with $\delta_0 = \delta_0 (\alpha, \delta_1^*, \gamma_2^*)$, and $\Lambda_0 = \frac{1}{|\gamma_0(\alpha, \delta_1^*, \gamma_2^*)|^2}$. 
\end{proof}

\subsection{Excluding small eigenvalues via far-field/core decomposition}

The argument here closely resembles that of \cite{AveryHolzerScheelKS}, but we adapt it for completeness. Fix $0 < \alpha < 1$. The goal of this section is to prove the following. 
\begin{prop}[Stability near the origin]\label{p: scf stability near origin}
	Fix $ 0 < \alpha < 1$. There exist $\delta_*(\alpha), \lambda_0(\alpha) > 0$ such that for all $0 < \delta < \delta_*(\alpha)$ and $|\lambda| \leq \lambda_0$ away from the essential spectrum, the equation
	\begin{align}
			(\Bsc(\delta) - K_\delta \lambda) \u = 0 \label{e: scf eigenvalue problem}
	\end{align}
	 has no solutions $\u \in H^2_{0, \etasc} (\R, \C^2)$ if $\Re \lambda \geq 0$, where $\etasc = \sqrt{\alpha}$. Moreover, there is no solution at $\lambda = 0$ which belongs to $L^\infty_{0, \etasc}(\R, \C^2)$.  
\end{prop}
Applying Proposition \ref{p: eigenvalue reduction}, we find that \eqref{e: scf eigenvalue problem} has a bounded solution if and only if we have a bounded solution to the reduced eigenvalue problem
\begin{align}
	\partial_{\xi \xi} \tilde{u} + a_1(\xi; \alpha, \lambda, \delta) \partial_\xi \tilde{u} + a_0 (\xi; \alpha, \lambda, \delta) \tilde{u} = \lambda \tilde{u}, \label{e: scf scalar eigenvalue problem}
\end{align}
where
\begin{align*}
	a_1 (\xi; \alpha, \lambda, \delta) &= \csc(\alpha) + f_v (u_\mathrm{scf}, \psi_v (u_\mathrm{scf}, \partial_\xi u_\mathrm{scf}; \alpha, \delta); \alpha, \delta) \tilde{\psi}_v^2 (u_\mathrm{scf}, \partial_\xi u_\mathrm{scf}; \alpha, \lambda, \delta), \\
	a_0 (\xi; \alpha, \lambda, \delta) &= f_u (u_\mathrm{scf}, \psi_v(u_\mathrm{scf}, \partial_\xi u_\mathrm{scf} ; \alpha, \delta)) + f_v (u_\mathrm{scf}, \psi_v (u_\mathrm{scf}, \partial_\xi u_\mathrm{scf}; \alpha, \delta); \alpha, \delta) \tilde{\psi}_v^1 (\u_\mathrm{scf}, \partial_\xi u_\mathrm{scf}; \alpha, \lambda, \delta),
\end{align*}
where we have suppressed the arguments of $u_\mathrm{scf} (\xi; \alpha, \delta^2), v_\mathrm{scf}(\xi; \alpha, \delta^2)$ for notational simplicity. 
\begin{lemma}\label{l: stability coefficient limits}
	We have 
	\begin{align*}
		\lim_{\xi \to \infty} a_1 (\xi; \alpha, \lambda, \delta) &= \csc (\alpha) = 2 \sqrt{\alpha}, \quad \lim_{\xi \to \infty} a_0 (\xi; \alpha, \lambda, \delta) = \alpha.
	\end{align*}
	Moreover, these limits are attained exponentially in $\xi$, with rate uniform in $\lambda, \delta$ small.
\end{lemma}
\begin{proof}
	Note from Proposition \ref{p: secondary CSC existence} that $(u_\mathrm{scf} (\xi; \alpha, \delta^2), v_\mathrm{scf}(\xi; \alpha, \delta^2)) \to (0, 1-\alpha)$ as $\xi \to \infty$, with exponential rate uniform in $\delta^2$ small. Recall also from Proposition \ref{p: existence regularization} that $\psi_v(0, 0; \alpha, \delta) = 1- \alpha$. Hence
	\begin{align*}
		\lim_{\xi \to \infty} f_v (u_\mathrm{scf}(\xi), \psi_v (u_\mathrm{scf}(\xi), \partial_\xi (\xi) u_\mathrm{scf}; \alpha, \delta); \alpha, \delta) = f_v (0, 1-\alpha; \alpha, \delta) = 0 
	\end{align*} 
	since $f_v(u,v) = - u$. This establishes the claim for $a_1$. By the same argument, the contribution to $a_0$ from its second term vanishes in the limit, and for the first term we have 
	\begin{align}
		\lim_{\xi \to \infty} f_u (u_\mathrm{scf}(\xi), \psi_v (u_\mathrm{scf} (\xi), \partial_\xi u_\mathrm{scf} (\xi); \alpha, \delta) = f_u (0, 1-\alpha) = \alpha,.  
	\end{align}
	as desired. 
\end{proof}
\begin{corollary}\label{c: stability e plus}
	The limiting eigenvalue problem at $+\infty$, 
	\begin{align}
		\partial_{\xi \xi} \tilde{u} + 2 \sqrt{\alpha} \partial_\xi \tilde{u} + \alpha \tilde{u} = \lambda \tilde{u},
	\end{align}
	admits a solution 
	\begin{align}
		e_+(\xi; \gamma) = e^{\nu_- (\gamma) \xi},
	\end{align}
	where $\nu_-(\gamma) = - \sqrt{\alpha} - \gamma$, and $\gamma = \sqrt{\lambda}$ with $\Re \gamma \geq 0$. 
\end{corollary}

To track eigenvalues possibly bifurcating from the essential spectrum, we make an ansatz which accounts for the loss of spatial localization of eigenfunctions as $\lambda = \gamma^2$ approaches the essential spectrum. That is, we fix $\tilde{\eta} > 0$ small and look for solutions to \eqref{e: scf scalar eigenvalue problem} via the far-field/core ansatz
\begin{align}
	\tilde{u} (\xi; \gamma) = \tilde{u}_c (\xi; \gamma) + \beta \chi_+(\xi) e_+(\xi; \gamma),
\end{align}
requiring $\tilde{u}_c \in H^2_{0, \eta} (\R, \C)$ with $\eta = \etasc + \tilde{\eta}$, so that if $|\gamma|$ is small, $\tilde{u}_c$ decays faster than $e_+(\xi; \gamma)$ as $\xi \to \infty$. Inserting this ansatz into \eqref{e: scf scalar eigenvalue problem} leads to the equation
\begin{align}
	0 = F_\mathrm{stab} (\tilde{u}_c, \beta; \gamma, \delta), \label{e: ff core eigenvalue problem}
\end{align}
where 
\begin{align*}
	F_\mathrm{stab} (\tilde{u}_c, \beta; \gamma, \delta) = \partial_{\xi \xi} [\tilde{u}_c + \beta \chi_+ e_+] + a_1 \partial_\xi [\tilde{u}_c + \beta \chi_+ e_+] + [a_0 - \gamma^2] [\tilde{u}_c + \beta \chi_+ e_+]. 
\end{align*}
We suppress the dependence on the fixed parameter $0 < \alpha < 1$. 
\begin{lemma}\label{l: stability well defined}
	Fix $\tilde{\eta} > 0$ small, and set $\eta = \etasc + \tilde{\eta}$. There exist $\gamma_0, \delta_0 > 0$ such that the map 
	\begin{align}
		F_\mathrm{stab}: H^2_{0, \eta} \times \C \times B(0, \gamma_0) \times (-\delta_0, \delta_0) \to L^2_{0, \eta} 
	\end{align}
	is well-defined, linear in $\tilde{u}_c$ and $\beta$, and $C^k$ in $\gamma$ and $\delta$. Moreover, the equation $(\Bsc(\delta) - K_\delta \gamma^2) \tilde{\mathbf{u}} = 0$ has a solution $\tilde{\mathbf{u}} \in H^2_{0, \etasc} (\R, \C^2)$ with $\gamma^2$ to the right of the essential spectrum if and only if \eqref{e: ff core eigenvalue problem} has a solution with $\Re \gamma \geq 0$. 
\end{lemma}	
\begin{proof}
	It follows from Lemma \ref{l: stability coefficient limits} and \ref{c: stability e plus} that $F_\mathrm{stab}$ preserves exponential localization of $\tilde{u}_c$, and hence is well-defined. Regularity in $\gamma$ and $\delta$ as well as equivalence to the original eigenvalue problem follow by combining Proposition \ref{p: eigenvalue reduction} with the arguments of \cite[Section 5]{PoganScheel}. 
\end{proof}
Since our ansatz naturally captures the loss of spatial localization associated to the essential spectrum, we recover Fredholm properties on exponentially weighted spaces, and can exploit this to track eigenvalues near the essential spectrum. 
\begin{prop}
	Let $\gamma_0$ and $\delta_0$ be as in Lemma \ref{l: stability well defined}. There exists a function $E : B(0, \gamma_0) \times (-\delta_0, \delta_0) \to \C$ which is $C^k$ in both arguments, such that:
	\begin{itemize}
		\item The equation $(\Bsc(\delta) - K_\delta \gamma^2) \tilde{\mathbf{u}} = 0$ has a solution $\tilde{\mathbf{u}} \in H^2_{0, \etasc} (\R, \C^2)$ with $\gamma^2$ to the right of the essential spectrum if and only if $E(\gamma, \delta) = 0$, with $\Re \gamma > 0$. 
		\item The equation $\Bsc(\delta) \tilde{\mathbf{u}} = 0$ has a solution $\tilde{\mathbf{u}} \in L^\infty_{0, \etasc} (\R, \C^2)$ if and only if $E(0, \delta) = 0$. 
	\end{itemize}
\end{prop}
\begin{proof}
	This follows from a Lyapunov-Schmidt reduction. Note that $D_{\tilde{u}_c} F_\mathrm{stab}(0, 0; 0, 0) = \Akpp$, and recall from Section \ref{s: existence} that $\Akpp : H^2_{0, \eta} \to L^2_{0, \eta}$ is Fredholm with index -1. More precisely, $\Akpp$ has trivial kernel and one-dimensional cokernel (see e.g. \cite[Section 2]{AveryGarenaux} for Fredholm properties of pulled Fisher-KPP fronts), spanned by some function $\phi_0$. Let $P_0$ denote the $L^2$-orthogonal projection onto the range of $\Akpp : H^2_{0, \eta} \to L^2_{0, \eta}$. We may then decompose the equation $F_\mathrm{stab}(\tilde{u}_c, \beta; \gamma, \delta) = 0$ as 
	\begin{align*}
		\begin{cases}
			P_0 F_\mathrm{stab} (\tilde{u}_c, \beta; \gamma, \delta) = 0 \\
			\langle F_\mathrm{stab}(\tilde{u}_c, \beta; \gamma, \delta), \phi_0\rangle = 0. 
		\end{cases}
	\end{align*}
	This system has a trivial solution $(\tilde{u}_c, \beta; \gamma, \delta)$. The linearization of the first equation with respect to $\tilde{u}_c$ at this trivial solution is $P_0 \Akpp$, which is invertible by construction. Hence, we may use the implicit function theorem to solve the first equation for $\tilde{u}_c (\beta; \gamma, \delta)$. Moreover, since the equation is linear in both $\tilde{u}_c$ and $\beta$, it follows that we must have $\tilde{u}_c (\beta; \gamma, \delta) = \beta \hat{u}_c (\gamma, \delta)$ for some $\hat{u}_c (\gamma, \delta) \in H^2_{0, \eta}$. Inserting this into the second equation, we find the scalar equation
	\begin{align*}
		 0 = \beta \langle [\partial_{\xi \xi} + a_1 \partial_x + a_0 - \gamma^2][\hat{u}_c + \chi_+ e_+], \phi_0 \rangle. 
	\end{align*}
	Hence, we cancel the factor of $\beta$, and define 
	\begin{align}
		E(\gamma, \delta) = \langle [\partial_{\xi \xi} + a_1 \partial_x + a_0 - \gamma^2][\hat{u}_c + \chi_+ e_+], \phi_0 \rangle. 
	\end{align}
	Regularity of $E(\gamma, \delta)$ follows from Lemma \ref{l: stability well defined}. 
\end{proof}
\begin{proof}[Proof of Proposition \ref{p: scf stability near origin}]
	We have shown that we have a solution to the original eigenvalue problem with $\lambda = \gamma^2$ to the right of the essential spectrum if and only if we have $E(\gamma, \delta) = 0$, with $\Re \gamma \geq 0$. Since $E$ is continuous in both arguments, to rule out unstable eigenvalues near the origin, it then suffices to show that $E(0, 0) \neq 0$. We find
	\begin{align*}
		E(0,0) = \langle \Akpp ( \hat{u}_c(0,0) + \chi_+ e_+(\cdot; 0)), \phi_0 \rangle. 
	\end{align*}
	Since $\hat{u}_c \in H^2_{0, \eta}$ is exponentially localized, we move $\Akpp$ to the other side of the inner product in this term as its adjoint, obtaining 
	\begin{align*}
		\langle \Akpp \hat{u}_c (0,0), \phi_0 \rangle = \langle \hat{u}_c(0,0), \Akpp^* \phi_0 \rangle = 0
	\end{align*}
	since $\phi_0$ is in the kernel of $\Akpp^* : H^2_{0, -\eta} \to L^2_{0, -\eta}$. Noting that $e_+(\xi; 0) = 1$, we then have
	\begin{align*}
		E(0,0) = \langle \Akpp \chi_+, \phi_0 \rangle. 
	\end{align*}
	It then follows from the proof of Lemma \ref{l: existence joint invertibility} that $E(0,0) \neq 0 $, as desired. 
\end{proof}

We are now ready to complete the proof of Proposition \ref{p: secondary csc spectral stability} by combining Propositions \ref{p: large lambda stability} and \ref{p: scf stability near origin}, together with a standard argument excluding eigenvalues in the intermediate $|\lambda|$ regime, to establish marginal spectral stability of the secondary CSC fronts. 
\begin{proof}[Proof of Proposition \ref{p: secondary csc spectral stability}]
	By Proposition \ref{p: large lambda stability}, there exists $\Lambda_0, \delta_0 > 0$ such that the equation $(\Bsc(\delta) - K_\delta \lambda) \tilde{\mathbf{u}} = 0$ has no bounded solutions if $|\delta| < \delta_0$ and $|\lambda| \geq \Lambda_0$ with $\Re \lambda \geq 0$. By Proposition \ref{p: scf stability near origin}, there exist $\lambda_0$ and $\delta_* > 0$ such that we have no bounded solutions to the same equation with $|\lambda| \leq \lambda_0, |\delta| \leq \delta_*$, and $\lambda$ away from the negative real axis. To complete the proof of Proposition \ref{p: secondary csc spectral stability}, it only remains to exclude eigenvalues in the intermediate region $\lambda_0 < |\lambda| < \Lambda_0$, $\Re \lambda \geq 0$. This follows, for instance, from spectral stability of the Fisher-KPP front at $\delta = 0$, together with robustness of exponential dichotomies. 
\end{proof}

\section{Tracking the total tumor mass --- heuristics and numerics}\label{s: mass}
We now explore implications of our predictions for the spreading speed on the dynamics of the total cancer mass, which is well defined when we consider \eqref{e: eqn} in a bounded domain $x \in [0, L]$. We use Neumann boundary conditions $u_x = v_x = 0$ at $x = 0$, and Dirichlet boundary conditions $u = v = 0$. If $L$ is sufficiently large, then the spreading dynamics in the bounded domain should be well approximated by the unbounded domain limit, until the front interface hits the right boundary \cite{AveryDedinaSmithScheel}. In this setting, we define the total cancer mass by 
\begin{align*}
	M(\tau; \alpha) = \int_0^L u(x, \tau; \alpha) + v(x, \tau; \alpha) \, dx,
\end{align*}
and are interested in determining when the cancer mass is increasing or decreasing in $\alpha$. 

We consider spreading from step function initial conditions,
\begin{align*}
	u_0 (x) = \begin{cases}
		1, & 0 \leq x \leq x_0, \\
		0, & x_0 < x \leq L,
	\end{cases}
	\qquad
	v_0 (x) = \begin{cases}
		1-\alpha, & 0 \leq x \leq x_0, \\
		0, & x_0 < x \leq L,
	\end{cases}
\end{align*}
for some fixed $0 < x_0 < L$, but expect similar results to hold for more general initial conditions. To make a heuristic prediction, informed by our rigorous results, we crudely approximate $u$ and $v$ with piecewise constant functions moving with the appropriate spreading speeds. 

\paragraph{Staged invasion regime.} Hence, in the staged invasion regime $0 < \alpha < \frac{1}{1+\eps}$, we define the approximate front positions
\begin{align*}
	\tilde{\sigma}_\mathrm{sc}(\tau; \alpha) &= \min (x_0 + \csc(\alpha) \tau, L), \\
	\tilde{\sigma}_\mathrm{pt}(\tau; \alpha, \eps) &= \min (x_0 + \cpt(\alpha, \eps) \tau, L),
\end{align*}
and make the approximations
\begin{align*}
	u(x,\tau; \alpha) \approx \begin{cases}
		1, &0 \leq x < \tilde{\sigma}_\mathrm{sc}(\tau; \alpha), \\
		0, & \tilde{\sigma}_\mathrm{sc}(\tau; \alpha) \leq x \leq L, 
	\end{cases}
\end{align*}
and
\begin{align*}
	v(x,\tau; \alpha) \approx 
	\begin{cases}
		0, & 0 \leq x < \tilde{\sigma}_\mathrm{sc}(\tau; \alpha), \\
		1-\alpha, & \tilde{\sigma}_\mathrm{sc}(\tau; \alpha) \leq x \leq \tilde{\sigma}_\mathrm{pt}(\tau; \alpha, \eps), \\
		0, &  \tilde{\sigma}_\mathrm{pt}(\tau; \alpha, \eps) < x \leq L. 
	\end{cases}
\end{align*}
This leads to the approximation for the total cancer mass
\begin{align}
	M_\mathrm{app}(\tau; \alpha, \eps) = \tilde{\sigma}_\mathrm{sc} (\tau; \alpha) + (1-\alpha) [\tilde{\sigma}_\mathrm{pt}(\tau; \alpha, \eps) - \tilde{\sigma}_\mathrm{sc}(\tau; \alpha)]. \label{e: Mapp}
\end{align}
Whether $M(\tau;\alpha, \eps)$ is increasing or decreasing in $\alpha$ then depends on whether either the primary or secondary front has reached the boundary. 

If neither front has reached the boundary, then $\tilde{\sigma}_\mathrm{sc} (\tau; \alpha) = x_0 + 2 \sqrt{\alpha} \tau$ and $\tilde{\sigma}_\mathrm{pt}(\tau; \alpha, \eps) = x_0 + 2 \sqrt{\frac{1-\alpha}{\eps}} \tau$, and a short calculation leads to 
\begin{align*}
	\partial_\alpha M_\mathrm{app}(\tau; \alpha, \eps) = \frac{3}{2} [\csc(\alpha) - \cpt(\alpha, \eps)] \tau < 0, 
\end{align*}
since $\cpt(\alpha, \eps) > \csc(\alpha)$ in the staged invasion regime. Hence, when we are in the staged invasion regime and neither front has yet reached the boundary, the total cancer mass is decreasing in the TC death rate $\alpha$, as one might hope. 

	\begin{figure}
	\centering
	\begin{subfigure}{0.495\textwidth}
		\includegraphics[width=1\textwidth]{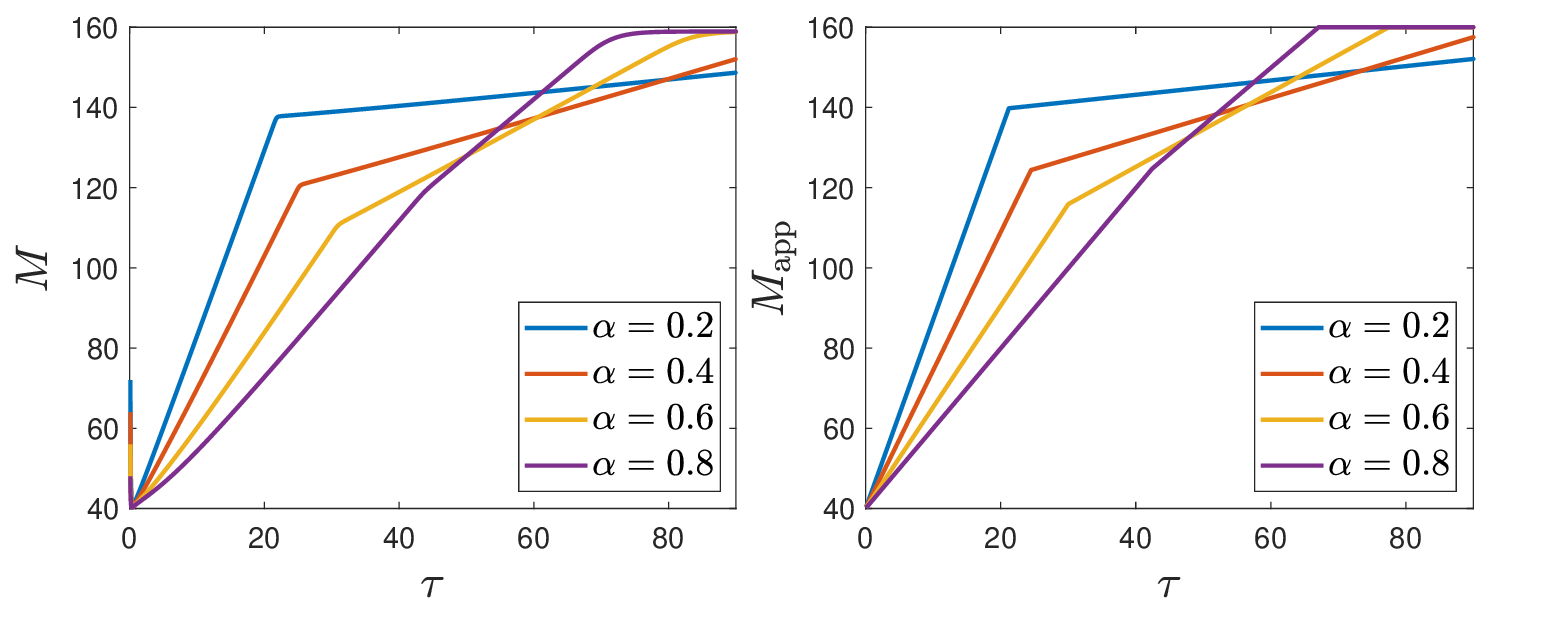}
	\end{subfigure}
	\hfill
	\begin{subfigure}{0.495\textwidth}
		\includegraphics[width=1\textwidth]{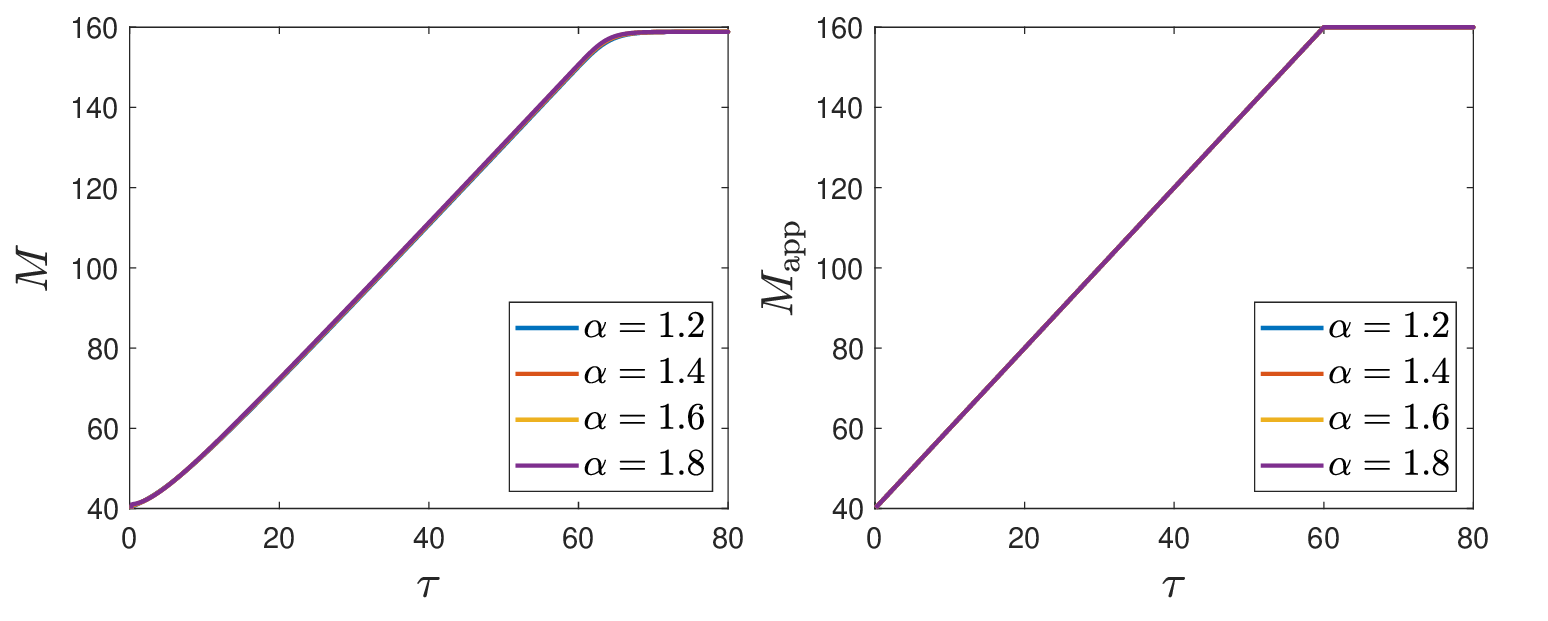}
	\end{subfigure}
	\caption{Left two panels: total cancer mass in the staged invasion regime, with $\eps = 0.1$ and $\alpha$ ranging from $0.2$ to $0.8$, measured via direct simulation (far left) and approximated via \eqref{e: Mapp} (center left). As predicted, the cancer mass is initially decreasing in $\alpha$, but then starts to increase shortly after the primary TC front hits the right boundary. Right two panels: total cancer mass in the TC extinction regime, with $\eps = 0.1$ and $\alpha$ ranging from $1.2$ to $1.8$. As predicted, the total cancer mass is essentially independent of $\alpha$ in this regime.}
	\label{f: mass}
\end{figure}

If the primary front has reached the boundary, so that $\tilde{\sigma}_\mathrm{pt}(\tau, \alpha, \eps) = L$ but $\tilde{\sigma}_\mathrm{sc} (\tau; \alpha) = x_0 + 2 \sqrt{\alpha} \tau$, then from a short calculation we find
\begin{align*}
	\partial_\alpha M_\mathrm{app}(\tau; \alpha, \eps) = x_0 + 3 \sqrt{\alpha} \tau - L. 
\end{align*}
Hence, $M_\mathrm{app}$ remains decreasing in $\alpha$ for $\tau$ such that $x_0 + 3 \sqrt{\alpha} \tau < L$, but then starts to \emph{increase} with $\alpha$ for $\tau$ such that $x_0 + 2 \sqrt{\alpha} \tau < L < x_0 + 3 \sqrt{\alpha} \tau$.

Once both fronts have reached the boundary, $\tilde{\sigma}_\mathrm{pt} = \tilde{\sigma}_\mathrm{pc} = L$, then the total cancer mass is constant, $M_\mathrm{app}(\tau, \alpha, \eps) \equiv L$.

In summary, in the staged invasion regime $0 < \alpha < \frac{1}{1+\eps}$: 
\begin{itemize}
	\item The total cancer mass is decreasing with $\alpha$ until $x_0 + \cpt(\alpha, \eps) \tau = L$.
	\item The total cancer mass remains decreasing in $\alpha$ until $x_0 + 3 \sqrt{\alpha} \tau = L$. 
	\item The total cancer mass is then increasing in $\alpha$ for $x_0 + 2 \sqrt{\alpha} \tau < L < x_0 + 3 \sqrt{\alpha} \tau$, that is, until the CSC front reaches the right boundary. 
\end{itemize}
This approximation is well confirmed by comparison to numerical simulations; see Figure \ref{f: mass}. These results highlight the importance of transient dynamics for discussing the tumor invasion paradox in this setting. 

\paragraph{TC extinction regime.} In the TC extinction regime $\alpha > \frac{1}{1+\eps}$, CSCs spread faster than TCs, and so we expect the solution to be dominated by CSCs. Hence, we define the approximate primary CSC front position
\begin{align*}
	\tilde{\sigma}_\mathrm{pc} (\tau) = \min(x_0 + 2\tau, L), 
\end{align*}
and make the approximations
\begin{align*}
	v(x,\tau; \alpha, \eps) \approx 0, \qquad u(x, \tau; \alpha, \eps) \approx \begin{cases}
		1, & 0 < x < \tilde{\sigma}_\mathrm{pc}(\tau) \\
		0, & \tilde{\sigma}_\mathrm{pc}(\tau) \leq x \leq L,
	\end{cases}
\end{align*}
leading to the simple approximation $M_\mathrm{app}(\tau, \alpha, \eps) = \min(2\tau, L)$ for the total cancer mass. In particular, in the TC extinction regime, the total cancer mass is roughly independent of the tumor death rate $\alpha$, which is corroborated by numerical simulations in Figure \ref{f: mass}.  
\section{Discussion} \label{s: discussion}
We have shown that in the $0 < \eps \ll 1$ limit of \eqref{e: eqn}, the spreading dynamics are governed by pulled (that is, linearly determined) invasion fronts, with predictions for the invasion speeds in Theorems \ref{t: primary invasion} through \ref{t: primary CSC invasion}. In the staged invasion regime $0 < \alpha < \frac{1}{1+\eps}$, one sometimes finds a tumor invasion paradox, where some measure of the tumor growth (either the total cancer mass or the spreading of CSC cells) is enhanced when the TC death rate $\alpha$ is increased. 

One limitation of the model \eqref{e: eqn} is that in the long time, the cancer cell population is dominated by CSCs, which is in contrast to experimental evidence that cancer stem cells typically make up a small fraction of a given tumor. This could be remedied by modifying \eqref{e: eqn} to allow for different carrying capacities of TCs and CSCs, which would be an interesting avenue for future work. As parameters representing these carrying capacities are varied, the relevant fronts may at some points become pushed, or nonlinearly determined. The methods of \cite{avery22} should allow one to efficiently explore this parameter space numerically, determining when the fronts are pushed or pulled. 

Our method for establishing linearly determinacy of invasion fronts is generally applicable to singularly perturbed two-component reaction diffusion systems. The key ingredients are that the formally reduced equation at $\eps = 0$ is of Fisher-KPP type, and that its fronts live within a normally hyperbolic slow manifold in the traveling wave formulation. 

	\bibliographystyle{abbrv}
	\bibliography{tumorbib}
\end{document}